\documentclass[11pt,a4paper]{article}
\usepackage{amsmath,amsthm,amssymb,amsfonts,hyperref,mathrsfs,bm,indentfirst,graphicx,enumitem,cases}
\usepackage[left=3.1cm,right=3.1cm,top=2cm,bottom=2cm]{geometry}

\newtheorem{theorem}{Theorem}[section]
\newtheorem{corollary}{Corollary}[section]
\newtheorem{definition}{Definition}[section]
\newtheorem{lemma}{Lemma}[section]
\newtheorem{proposition}{Proposition}[section]
\newtheorem{remark}{Remark}[section]

\newcommand{\bu}{\bm{u}}
\newcommand{\bsigma}{\bm{\sigma}}
\newcommand{\bbeta}{\bm{\eta}}
\newcommand{\teta}{\tilde{\bbeta}}
\newcommand{\bd}{\bm{d}}

\newcommand{\bv}{\bm{v}}
\newcommand{\bV}{\bm{V}}
\newcommand{\bq}{\bm{q}}
\newcommand{\bS}{\bm{S}}
\newcommand{\bD}{\bm{D}}

\newcommand{\bX}{\bm{X}}
\newcommand{\bU}{\bm{U}}
\newcommand{\bnu}{\bm{\nu}}
\newcommand{\btau}{\bm{\tau}}
\newcommand{\bpsi}{\bm{\psi}}
\newcommand{\bphi}{\bm{\phi}}
\newcommand{\bA}{\bm{A}}
\newcommand{\bB}{\bm{B}}
\newcommand{\bvs}{\bv^{*}}
\newcommand{\bVs}{\bV^{*}}
\newcommand{\bw}{\bm{w}}
\newcommand{\bL}{\bm{L}}
\newcommand{\bH}{\bm{H}}
\newcommand{\bC}{\bm{C}}
\newcommand{\bW}{\bm{W}}
\newcommand{\br}{\bm{r}}

\newcommand{\dt}{\mathrm{d} t}
\newcommand{\dz}{\mathrm{d} z}
\newcommand{\ds}{\mathrm{d} s}
\newcommand{\rd}{\ \mathrm{d}}
\newcommand{\pt}{\partial_{t}}
\newcommand{\ptt}{\partial_{tt}}

\newcommand{\nuS}{\bnu_{S}}
\newcommand{\nuF}{\bnu_{F}}
\newcommand{\tauF}{\btau_{F}}

\newcommand{\JNn}{J_{N}^{n}}
\newcommand{\JNN}{J_{N}}
\newcommand{\Jn}{J^{n}}
\newcommand{\JN}[1]{J_{N}^{#1}}
\newcommand{\Jeta}{J^{\bbeta}}

\newcommand{\onehalf}{\frac{1}{2}}

\newcommand{\TN}{\mathcal{T}_{N}}
\newcommand{\ad}[1]{\begin{aligned} #1 \end{aligned}}

\newcommand{\abs}[1]{\left\vert #1 \right\vert}
\newcommand{\norm}[1]{\left\Vert #1 \right\Vert}
\newcommand{\normg}[1]{\left\Vert #1 \right\Vert_{\bL^{2}(\Gamma)}^{2}}

\newcommand{\normos}[1]{\left\Vert #1 \right\Vert_{\bL^{2}(\Omega_{S})}^{2}}
\newcommand{\normgh}[1]{\left\Vert #1 \right\Vert_{\bH_0^{2}(\Gamma)}^{2}}

\newcommand{\EN}[1]{E_{N}^{#1}}
\newcommand{\DN}[1]{D_{N}^{#1}}
\newcommand{\un}[1]{\bu_{N}^{#1}}
\newcommand{\dn}[1]{\bd_{N}^{#1}}

\newcommand{\etan}[1]{\bbeta_{N}^{#1}}
\newcommand{\tetan}[1]{\teta_{N}^{#1}}
\newcommand{\vn}[1]{\bv_{N}^{#1}}
\newcommand{\VN}[1]{\bV_{N}^{#1}}
\newcommand{\AN}[1]{\bA_{N}^{#1}}

\newcommand{\wn}[1]{\bw_{N}^{#1}}
\newcommand{\uNn}{\bu_{N}}
\newcommand{\dNn}{\bd_{N}}

\newcommand{\etaNn}{\bbeta_{N}}
\newcommand{\tetaNn}{\teta_{N}}
\newcommand{\vNn}{\bv_{N}}
\newcommand{\vsNn}{\bvs_{N}}
\newcommand{\VNn}{\bV_{N}}
\newcommand{\VsNn}{\bVs_{N}}
\newcommand{\ANn}{\bA_{N}}

\newcommand{\wNn}{\bw_{N}}
\newcommand{\suNn}{\{\bu_{N}\}_{N \in \mathbb{N}}}
\newcommand{\sdNn}{\{\bd_{N}\}_{N \in \mathbb{N}}}

\newcommand{\setaNn}{\{\bbeta_{N}\}_{N \in \mathbb{N}}}

\newcommand{\svNn}{\{\bv_{N}\}_{N \in \mathbb{N}}}
\newcommand{\sVNn}{\{\bV_{N}\}_{N \in \mathbb{N}}}

\newcommand{\Dt}{\Delta t}
\newcommand{\as}[2]{a_{S}( #1 , #2 )}
\newcommand{\V}{\mathcal{V}}
\newcommand{\Ve}{\mathcal{V}^{n}}
\newcommand{\W}{\mathcal{W}}
\newcommand{\Q}{\mathcal{Q}}
\newcommand{\QE}{\mathcal{Q}^{\bbeta}(0,T)}

\newcommand{\XE}{\mathcal{X}^{\bbeta}(0,T)}
\newcommand{\XM}{\mathcal{X}_{\mathrm{max}}}
\newcommand{\cH}{\mathcal{H}}
\newcommand{\cA}{\mathcal{A}}
\newcommand{\cD}{\mathcal{D}}
\newcommand{\cF}{\mathcal{F}}
\newcommand{\intt}{\int_{0}^{T}}
\newcommand{\inttt}{\int_{n\Dt}^{(n+1)\Dt}}
\newcommand{\inner}[2]{\left\langle #1 , #2 \right\rangle}
\newcommand{\LA}{\mathcal{L}_{1}}
\newcommand{\LB}{\mathcal{L}_{2}}
\newcommand{\LE}{\mathcal{L}_{e}}
\newcommand{\sA}{\mathscr{A}}
\newcommand{\sN}{\mathscr{N}}
\newcommand{\sF}{\mathscr{F}}
\newcommand{\sH}{\mathscr{H}}
\newcommand{\rv}[1]{\left. #1 \right|}
\newcommand{\innerl}[1]{\left\langle \LE #1 , #1 \right\rangle}
\newcommand{\tvNn}{\tilde{\bv}_N}
\newcommand{\stvNn}{\{\tilde{\bv}_{N}\}_{N \in \mathbb{N}}}
\newcommand{\btv}{\tilde{\bv}}
\newcommand{\bbS}{\mathbb{S}}

\newcommand{\qN}{\bq_{N}}
\newcommand{\phiN}{\bphi_{N}}
\newcommand{\psiN}{\bpsi_{N}}

\newcommand{\tu}{\tilde{{\bu}}}
\newcommand{\tun}{{\tu}_N}

\newcommand{\tvn}{{\bv}_N^*}

\newcommand{\tVn}{{\bV}_N^*}

\newcommand{\stun}{{\bu}_N^*}
\newcommand{\OM}{\Omega_{\rm max}}
\newcommand{\TM}{\widetilde{\bm{M}}}
\newcommand{\TG}{\widetilde{\bm{G}}}
\newcommand{\by}{\bm{y}}
\newcommand{\bz}{\bm{z}}
\newcommand{\nuFN}{\bnu_{F,N}}
\newcommand{\tauFN}{\btau_{F,N}}
\newcommand{\bvphi}{\bm{\varphi}}

\newcommand{\hun}[1]{\hat{\bu}_N^{#1}}

\DeclareMathOperator*{\meas}{meas.}
\DeclareMathOperator*{\tr}{tr}

\allowdisplaybreaks[4]

\begin{document}

\title{Existence of weak solutions to a generalized nonlinear multi-layered fluid-structure interaction problem with the Navier-slip boundary conditions}

\author{Wenjun Liu$^{a}$\footnote{Corresponding author. Email address: wjliu@nuist.edu.cn (W. Liu), Yadong.Liu@ur.de (Y. Liu), jun.yu@uvm.edu (J. Yu) },
	Yadong Liu$^{a,b}$ and Jun Yu$^{c}$\\
	a.  School of Mathematics and Statistics, Nanjing University of Information \\ Science and Technology, Nanjing 210044, China\\
	b. Fakult\"at f\"ur Mathematik, Universit\"at Regensburg, \\ 93040 Regensburg, Germany\\
	c. Department of Mathematics and Statistics, University of \\ Vermont, Burlington, USA}

\date{\today}
\maketitle

\begin{abstract}
	We consider a fluid-structure interaction problem with Navier-slip boundary conditions in which the fluid is considered as a non-Newtonian fluid and the structure is described by a nonlinear multi-layered model. The fluid domain is driven by a nonlinear elastic shell and thus is not fixed. To simplify the problem, we map the moving fluid domain into a fixed domain by applying an arbitrary Lagrange Euler mapping. Unlike the classical method by which we can consider the problem as its entirety, we utilize the time-discretization and split the problem into a fluid subproblem and a structure subproblem by an operator splitting scheme. Since the structure subproblem is nonlinear, Lax-Milgram lemma does not hold. Here we prove the existence and uniqueness by means of the traditional semigroup theory. Noticing that the Non-Newtonian fluid possesses a $ p- $Laplacian structure, we show the existence and uniqueness of solutions to the fluid subproblem by considering the Browder-Minty theorem. With the uniform energy estimates, we deduce the weak and weak* convergence respectively. By a generalized Aubin-Lions-Simon Lemma proposed by Muha and Cani\'c [J. Differential Equations {\bf 266} (2019), 8370--8418], we obtain the strong convergence. Finally, we construct the test functions and pass the approximate weak formulation to the limit as time step goes to zero with the convergence results.
\end{abstract}


\maketitle
\section{Introduction}
	This paper deals with a generalized multi-layered fluid-structure interaction problem with Navier-slip boundary conditions, which consists of a generalized fluid, a nonlinear thin structure and a thick structure.
	
\subsection{Model description}
	We consider a half cylindrical fluid domain $ \Omega_{F}^{\bbeta}(t) $ composed by a moving boundary $ \Gamma^{\bbeta}(t) $ and three rigid boundaries $ \Gamma_{\rm in} \cup \Gamma_{\rm out} \cup \Gamma_{\rm b} =: \Sigma $, i.e., $ \partial \Omega_{F}^{\bbeta}(t) = \Gamma^{\bbeta}(t) \cup \Sigma $ (see Figure \ref{geometry}). The displacement of thin structure is depicted by $ \bbeta : [0,T) \times \Gamma \rightarrow \mathbb{R}^2 $. Assume that the length of fluid domain is $ L $ and the reference radius is $ r = 1 $. Then we have the parameterized fluid domain as
	$$
		\Omega_{F}^{\bbeta}(t) = \left\{ (z,r) \in \mathbb{R}^2: z \in (0,L), r \in (0, 1 + \bbeta \cdot \nu) \right\},
	$$
	where $ \nu = (0,1) $ and the interface boundary as
	$$
		\Gamma^{\bbeta}(t) = \left\{ (z,r) \in \mathbb{R}^2: z \in (0,L), r = 1 + \bbeta \cdot \nu  \right\}.
	$$
	
	\begin{figure}[h]
		\centering
		\includegraphics[width = 0.6 \textwidth]{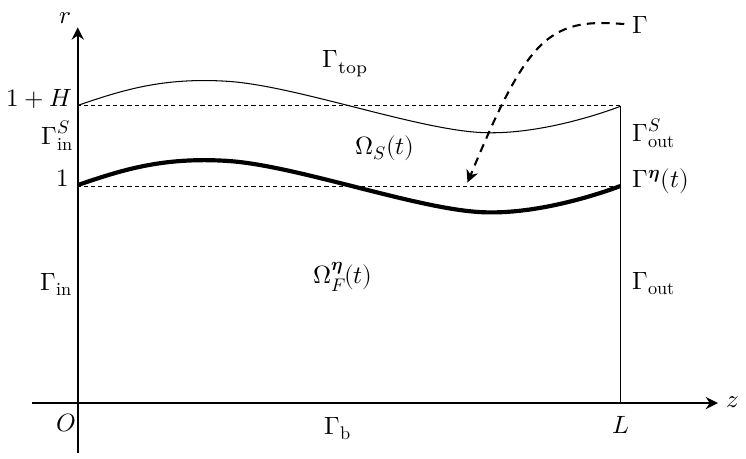}
		\caption{Geometry of fluid-structure interaction problem}
		\label{geometry}
	\end{figure}
	Subsequently, we model fluid motion by the two dimensional incompressible Navier-Stokes equations in $ \Omega_{F}^{\bbeta}(t) $:
	\begin{equation}\label{NS}
		\left.
			\ad{
				\pt \bu + \bu \cdot \nabla \bu & = \nabla \cdot \bsigma \\
				\nabla \cdot \bu & = 0
			}
		\right\}
		\text{ in } (0,T) \times \Omega_{F}^{\bbeta}(t),
	\end{equation}
	where $ \bu $ is the fluid velocity, $ \bsigma = - \pi \mathbb{I} + 2 \bbS(\bD(\bu))$ denotes the stress tensor, $ \pi $ is the fluid pressure. $\bbS(\bD(\bu)) = \mu_{F} ( \bD(\bu) ) \bD(\bu)$ represents the viscous effects with viscosity $ \mu_{F} ( \bD(\bu) ) $, which is a nonlinear term. $\bD(\bu) = \onehalf(\nabla\bu + \nabla^{T}\bu) $ is the symmetric gradient.
	In this manuscript, we consider the non-Newtonian fluid whose viscosity is the so called `power law' proposed by Carreau in his Ph.D. Thesis (see also e.g., \cite{AM1973,Galdi}):
	$$
		\mu_{F} ( \bD(\bu) ) = (1+\abs{\bD(\bu)}^{2})^{\frac{p-2}{2}},\quad 2 < p < \infty.
	$$
	The associated initial date of problem \eqref{NS} is
	$$
		\bu(0, \cdot) = \bu_{0}.
	$$
	
	On the rigid part of the fluid domain boundaries, we have
	\begin{gather}
		\left.
		\ad{\bu \cdot \nuF & = 0 \\
			\partial_{\bnu} \bu_{\btau} & = 0
		}
		\right\}
		\text{ on } (0,T) \times \Gamma_{\rm b}, \\
		\left.
		\ad{\pi + \onehalf \abs{\bu}^{2} & = P_{\rm in/out}(t) \\
			\bu \cdot \tauF & = 0
		}
		\right\}
		\text{ on } (0,T) \times \Gamma_{\rm in/out},
	\end{gather}
	where $ \nuF $ and $ \tauF $ are outer normal and tangential vectors of fluid domain respectively, and $ P_{\rm in/out}(t) $ denotes $ P_{\rm in}(t) $ or $ P_{\rm out}(t) $, providing ``inlet'' or ``outlet'' boundary data.
	
	On the elastic part of fluid domain boundary, interaction between $ \bu $ and $ \bbeta $ occurs. Let $ \Gamma = \Gamma^{\bbeta}(0) = (0,L) $ be the Lagrangian domain. Then the thin structure elastodynamic problem is given by
	\begin{align}
		 \ptt \bbeta + \LE \bbeta + f (\bbeta) & = \bm{h}, && \text{ on } (0,T) \times \Gamma, \label{eta}\\
		\bbeta & = \bm{0}, && \text{ on } (0,T) \times \partial \Gamma,
	\end{align}
	where $ f $ is a nonlinear term that will be assigned later. $ \LE $ is a continuous, self-adjoint, coercive, linear operator defined on $ \left[ H_0^2(\Gamma) \right]^2 $ such that
	$$
		\inner{\LE \bbeta}{\bbeta} \geq \delta_0 \norm{\bbeta}_{\left[ H_0^2(\Gamma) \right]^2}^2,\quad \forall\ \bbeta \in \left[ H_0^2(\Gamma) \right]^2,
	$$
	with $ \delta_0 $ be a positive constant, where $ \inner{\cdot}{\cdot} $ is the duality pairing between $ \left[ H_0^2(\Gamma) \right]^2 $ and $ \left[ H^{-2}(\Gamma) \right]^2 $.
	
	The other side of thin structure is the thick structure with thickness $ H $. We define the thick structure domain as
	$$
		\Omega_{S} = (0,L) \times (1, 1+H),
	$$
	with boundary $ \partial \Omega_{S} = \Gamma \cup \Gamma_{\rm in/out}^S \cup \Gamma_{\rm top} $. By Lagrangian formulation, the motion of thick layer defined on $ \Omega_{S} $ is described by a linear elastic equation:
	\begin{align}
		& \ptt \bd = \nabla \cdot \bS, && \text{ in } (0,T) \times \Omega_{S}, \label{dS}
	\end{align}
	with boundary conditions:
	\begin{align}
		& \bd = \bm{0}, && \text{ on } (0,T) \times \Gamma_{\rm in/out}^{S}, \\
		& \bS \nuS = \bm{0}, && \text{ on } (0,T) \times \Gamma_{\rm top}, \label{Gammatop}
	\end{align}
	where $ \bd $ denotes the displacement of thick structure, $ \bS $ is the first Piola-Kirchhoff stress tensor given by $ \bS = 2 \mu_{S} \bD(\bd) + \lambda (\nabla \cdot \bd) \mathbb{I} $ and $ \mu_{S} > 0 $ is the elasticity of thick structure.

	Moreover, follows are the \textbf{coupled} conditions.
	\begin{itemize}
		\item The kinematic conditions:
		\begin{align}
			& \bu \cdot \nuF = \pt \bbeta \cdot \nuF \ (\text{Continuity of normal velocity on } \Gamma^{\bbeta}(t)), \\
			& ( \pt \bbeta - \bu ) \cdot \tauF = \alpha \bsigma \nuF \cdot \tauF \ (\text{Slip effect on } \Gamma^{\bbeta}(t)), \\
			& \bd = \bbeta \ (\text{Continuity of displacements on } \Gamma). \label{deta}
		\end{align}
		\item The dynamic coupling condition:
		\begin{equation*}
			\bm{h} = - J_F^{\bbeta} \bsigma \nuF - \bS \nuS, \ \text{ on } (0,T) \times \Gamma,
		\end{equation*}
	\end{itemize}
	where $ \alpha > 0 $ is the ratio constant of the slip effect and it is assumed to be suitably small for convenience in the proof of Lemma \ref{fluidpro}. $ J_F^{\bbeta} $ denotes the Jacobian of transformation from Eulerian to Lagrangian formulations.
	$ \bS $ is the stress acted on thin structure from thick structure	 and $ \nuS $ is the outer normal vector of thick structure. We note here that on the interface, $ \nuF = - \nuS $.
	
	In addition, problem \eqref{NS}--\eqref{deta} satisfies the initial conditions
	\begin{equation}
		\bu(0,\cdot) = \bu_{0},\quad \bbeta(0,\cdot) = \bbeta_{0},\quad \pt\bbeta(0,\cdot) = \bv_{0},\quad \bd(0,\cdot) = \bd_{0}, \quad \pt \bd(0,\cdot) = \bV_{0},
	\end{equation}
	and necessary compatibility conditions (see \cite{MC2016JDE})
	\begin{itemize}
		\item The initial fluid velocity must satisfy:
		\begin{equation}\label{u0}
			\ad{
				& \bu_{0} \in  L^2(\Omega_{F}^{0})^2, \  \nabla \cdot \bu_{0} = 0,  && \text{ in } \Omega_{F}^0, \\
				& \bu_{0} \cdot \nu = 0, && \text{ on } \Gamma_{\rm b}, \\
				& \bu_{0} \cdot \bnu_{0} = \bv_{0} \cdot \bnu_{0}, && \text{ on } \Gamma^0,
			}
		\end{equation}
		where $ \Omega_{F}^0 = \Omega_{F}^{\bbeta}(0) $, $ \Gamma^0 = \Gamma^{\bbeta}(0) $, $ \bnu_{0} = \bnu^{\bbeta}(0, \cdot) $.
		\item The initial domain must be such that there exists a diffeomorphism $ \bvphi^0 \in C^1(\overline{\Omega})_{F} $ such that
		\begin{equation}\label{phi0}
			\bvphi^0 (\Omega_{F}) = \Omega_{F}^0, \  \det \nabla \bvphi^0 > 0, \  \rv{\left(\bvphi - \bm{I}\right)}_{\Gamma} = \bbeta_{0},
		\end{equation}
		and the initial displacement $ \bbeta_{0} $ is such that
		\begin{equation}\label{eta0}
			\norm{\bbeta_{0}}_{H^{11/6}} \leq c, \text{ where } c \text{ is small}.
		\end{equation}
	\end{itemize}
	
\subsection{Motivation}
	In recent years, mathematical problems of fluid-structure interaction have been studied continuously. These problems arise from several applications in different fields, such as biomechanics, blood flow dynamics, aeroelasticity, hydroelasticity and so on. There are many research works that are investigating the problems from different aspects in the area of analysis and numerical simulations \cite{Bm2016,Bc2020}.
	For fluid-elastic interaction problem with strong solutions results, Beir\~{a}o da Veiga \cite{BDV2004} consider 2D fluid and 1D linear elastic model with periodic boundary conditions. They proved the local strong solutions by the linearization of fluid equation and fixed point theorem. Coutand and Shkoller showed the existence of a unique regular local solution of a 3D fluid-3D structure (linear \cite{CS2005} or quasi-linear \cite{CS2006} elasticity) immersed in the fluid. Then Cheng and Shkoller \cite{ChS2010} extended the strong solution results to nonliear elastic Koiter shell. Lequeurre \cite{Lequeurre2011,Lequeurre2013} expanded the results of Beir\~{a}o da Veiga. They obtained the existence of a unique, local in time, strong solution for any data, and showed the globally existence of strong solutions with small initial data. Later, Grandmont and Hillairet \cite{GH2016} considered a 2d fluid equation coupled with elastic beam equation and obatin the globally strong solution by means of a regularized system. Subsequently, Grandmont, Hillairet and Lequeurre \cite{GHL2019} generalized the beam equation in \cite{GH2016} with different parameters combination. They improved the existence result of strong solutions with the same regularity of initial data by a regularization method. More results concerning strong solution can be found in \cite{IKLT2014,IKLT2017,KT2012b,KTZ2009,QGY2020,QY2018}.
	
	For weak solutions of fluid-elastic interaction problems, there are also many interesting results. Chambolle, Desjardins, Esteban and Grandmont \cite{CDEG2005} investigated an interaction problem between an incompressible fluid and a viscoelastic plate in 3D. After that, Grandmont \cite{Grandmont2008} studied a 2D incompressible fluid interacted with a 1D elastic plate. They constructed a perturbed viscoelastic term in plate equation and analyzed the limiting problem. In \cite{Lengeler2014}, Lengeler and R\r{u}\v{z}i\v{c}ka took a nature method to discuss the compactness issue in a 3D fluid-Koiter shell interaction problem. All the results above were obtained by constructing a regularized system and making energy estimates so that to pass variables to the limit with compactness principle. At the same time, Muha and \v{C}ani\'{c} \cite{CGM2020,MC2013,MC2013ARMA,MC2014JDE,MC2016JDE} studied the existence of weak solutions to a series of different fluid-structure interaction problems involving incompressible viscous fluid. They came up with a numerical-like way to prove this type of problems inspired by the numerical scheme in \cite{GGCC2009}. Their methods includes taking Arbitrary Lagrangian Euler mapping (ALE) to fix moving fluid domain, splitting the problem by Lie operator splitting, constructing approximation solutions with the idea of time-discrete iterative solution and proving the existence of fluid subproblem and structure subproblem respectively, so as to show the existence of the weak solution according to the compactness principle. More specifically, they discussed the interaction between 2D case in \cite{MC2013ARMA}, while in \cite{MC2013,MC2015} 3D cylindrical case, and linear, nonlinear Koiter shell equations were studied respectively. In \cite{MC2014JDE}, they considered a more realistic model associated with human arteries vessel which contains multi-layers. The model was abstracted into a 2D fluid interacting with a multi-layered structure including a 1D thin and a 2D thick elastic structure. In \cite{MC2016JDE}, Muha and \v{C}ani\'{c} considered different boundaries and took Navier-slip boundary into account in the system, which will allow both longitudinal and tangential components of displacement. The system was simplified to a two-dimensional case for subsequent analysis. Later, Trifunovi\'{c} and Wang \cite{TW2018} combined this method with a hybrid approxiamtion scheme to deal with a 3D incompressible fluid coupled with a nonlinear plate equation. They used the Galerkin method for the structure subproblem and passed both spatial variable $ k $ and time discrete variable $ \Dt $ to the limit and obtained the existence of the weak solutions. Subsequently, they consider an interaction problem between a viscous fluid and a thermoelastic plate \cite{TW2020}. The latest work was \cite{CGM2020} done by \v{C}ani\'{c},  Gali\'{c} and Muha. They addressed a 3D nonlinear fluid-mesh-shell interaction problem with moving boundary, in which they added a net of 1D hyperbolic equations to model the elastodynamics of an elastic mesh of curved rods. The results improved the simpler problem \cite{CGLMTW2019} defined in a fixed fluid domain.
		
	In this paper, we consider a generalized multi-layered fluid-structure interaction problem with Navier-slip boundarys condition, in which 2D \textit{non-Newtonian} fluid is bounded on one side by a \textit{nonlinear} thin structure and a linear elastic thick structure, while the interaction of fluid and thin structure is driven by \textit{Navier-slip} effects. In all the studies mentioned above, the viscosity of fluid was treated as a constant. However, in nature, ideal fluid does not exist, which means the viscosity of the fluid decreases with the increase of the shear strain rate (pseudoplastic, shear thinning), while in other case it behaves just the opposite (Dilatant, shear thickening). Therefore, we want to study the fluid-structure interaction problem for non-Newtonian fluid and multi-layered structure in order to model blood flow in human artery. For the work of non-Newtonian fluid-structure interaction problems, we notice that Lengeler \cite{Lengeler2014} generalized the viscous Newtonian fluid in \cite{LR2014} to a non-Newtonian fluid interacting with a linear elastic Koiter shell. They introduced a shear-dependent viscosity, which obeys ``power law'', and resolved the issue of additional stress (non-Newtonain) limit. Finally, they used the regularized system to obtain the relative compact for $ p > \frac{3}{2} $. In \cite{HLN2016}, Hundertmark-Zau\v{s}kov\'{a}, Luk\'{a}\v{c}ov\'{a}-Medvi\v{d}ov\'{a} and Ne\v{c}asov\'{a} set $ p > 2 $ in power law, which means the fluid is shear-thickening, and investigated the existence of weak solutions by fixed point procedure. The techniques dealing with non-Newtonain limit can also be found in \cite{Galdi,MNRR1996,Wolf2007}.
	
\subsection{Methodology and features}
	\label{methodology}
	In present work, we analyze \eqref{NS}--\eqref{eta0} by the method proposed by Muha and \v{C}ani\'{c}. More specifically, there are the following steps:
	\begin{itemize}
		\item Applying the ALE mapping to the problem and obtaining the weak formulation in flxed reference fluid domain $ \Omega_{F} $, see Section \ref{ALE};
		\item Taking Lie splitting method to decompose system \eqref{NS}--\eqref{eta0} into a fluid subproblem and a structure subproblem, showing the existence and uniqueness for both subproblems in each time subinterval and deriving the uniform estimates, see Sections \ref{splitting}--\ref{unifromestimates};
		\item Concluding the weak and weak* convergences from the uniform boundedness, see Sections \ref{uniformboundedness}--\ref{weakconvergence};
		\item Combining the compactness Lemma and compact embeddings to derive the strong convergences of velocities, displacement and geometry parameters, see Section \ref{strongconvergence};
		\item Passing to the limit as $ N \rightarrow \infty $, see Section \ref{limitproblem}.
	\end{itemize}
	
	Besides the procedure of Muha and \v{C}ani\'{c}, we have the following features when solving our problem:
	\begin{enumerate}
		\item[(i)] Muha and \v{C}ani\'{c} used Lax-Milgram theorem to show the existence and uniqueness of subproblem in \cite{MC2013,MC2013ARMA,MC2014JDE,MC2016JDE}. However, in our paper, Lax-Milgram Theorem does not hold due to the nonlinear term $ f(\bbeta) $. Unlike \cite{TW2018}, in which Galerkin method was used, we prove the existence and uniqueness of solutions to system \eqref{subf} by means of the traditional semigroup theory \cite{Pazy1983}.
		\item[(ii)] Since the Non-Newtonian constitutive relation is nonlinear, we can not apply the standard Lax-Milgram Lemma. We notice the $ p- $structure of fluid subproblem and show that using the Browder-Minty theorem \cite[Theorem 9.14--1]{Ciarlet2013} for $ p > 2 $ works well in our problem.
		\item[(iii)] When we summarize the weak and weak* convergence, we get the $ L^p $ weak convergence for symmetric gradient $ \bD^{\etaNn}(\uNn) $ and $ L^q $ weak convergence ($ \frac{1}{p} + \frac{1}{q} = 1 $) for $ \bbS (\bD^{\etaNn}(\uNn)) $. Their limit cannot be deduced directly due to the variant subscript $ \etaNn $ (related to fluid domain) and nonlinearity from $ \bbS $. We modified the proof of Proposition 7.6 in \cite{MC2014JDE} by introducing the localized Minty's Trick to obtain the limit.
	\end{enumerate}
	
\section{Preliminaries and main result}
	Since problem \eqref{NS}--\eqref{eta0} is defined in a moving fluid domain which is part pf unknowns, we can not define its weak solutions directly. To overcome this difficult, we introduce an Arbitrary Lagrangian Eulerian (ALE) mapping, which is common in numerical simulations of fluid-structure interaction problems. This maps our problem to a fixed domain so that we can carry out our analysis. There are many applications of ALE mapping in fluid-structure problems, see e.g., \cite{CGM2020,MC2013,MC2013ARMA,MC2014JDE,MC2015,MC2016JDE,MC2019JDE}.
	
	Before performing the ALE mapping in Section \ref{ALE}, we provied some useful facts and assumptions in Section \ref{facts} and the energy differential inequality associated with \eqref{NS}--\eqref{eta0} in Section \ref{energyi}. Sections \ref{transformation}--\ref{spacesetting} are devoted to transformation and space settings, respectively. Finally, our main result (Theorem \ref{mainresult}) is presented in Section \ref{weakmain}.
	
\subsection{Some useful facts and assumptions} \label{facts}
	\begin{lemma}
		It can be easily checked that for $ p > 2 $, $ \bbS $ satisfies \cite{Galdi}
		\begin{enumerate}
			\item Coercivity:
			\begin{equation}\label{Coercivity}
				\bbS(\bD) : \bD \geq \kappa_{1} \bD^{p} - \kappa_{2};
			\end{equation}
			\item Growth:
			\begin{equation}\label{Boundedness}
				\abs{ \bbS(\bD) } \leq \kappa_{3} \left( \abs{\bD}^{p-1} + 1 \right);
			\end{equation}
			\item Monotonicity:
			\begin{equation}\label{Monotonicity}
				\left( \bbS(\bD_{1}) - \bbS(\bD_{2}) \right) : \left( \bD_{1} - \bD_{2} \right) > 0, \text{ if } \bD_{1} \neq \bD_{2}.
			\end{equation}
		\end{enumerate}
		Here, $ \bD_{1} = \bD(\bu_{1}) $, $ \bD_{2} = \bD(\bu_{2}) $ and the notations $ \kappa_{i},\ i = 1,2,3 $ are constants, depending at most on $ p $, such that $ \kappa_{i} > 0,\ i = 1,3 $ and $ \kappa_{2} \in \mathbb{R} $. We remark that these three results will used in the proof of existence of the system.
	\end{lemma}	
	
	\textbf{Assumptions:}
	
%
%
%
	\begin{enumerate}[label=(f\arabic*)]
		\item \label{f1} $ f $ is locally Lipschitz from $ \bH^{2-\epsilon} $ into $ H^{2} $, namely, there exists a constant $ C_R > 0 $ suitably small such that
		$$
			\norm{f(\bbeta_{1}) - f(\bbeta_{2})}_{H^{2}(\Gamma)} \leq C_R \norm{\bbeta_{1} - \bbeta_{2}}_{\bH^{2-\epsilon}(\Gamma)},
		$$
		for some $ \epsilon > 0 $ and for every $ \norm{\bbeta_{i}}_{\bH^{2-\epsilon}(\Gamma)} \leq R \  (i = 1,2) $.
	\end{enumerate}

	\begin{remark}
		Since the bound of approximate solutions $ \etaNn $ is obtained in $ \bH_0^2(\Gamma) $, the assumption \ref{f1} is used to pass to the limit of nonlinear term $ f(\etaNn) $. We note here that this assumption is a little bit stronger for the nonlinear term $ f(\bbeta) $, while in \cite{TW2018}, Trifunovi\'{c} and Wang made two weaker Lipschitz assumptions for $ f(\bbeta) $ from $\bH_0^{2-\epsilon}(\Gamma) $ into $ H^{-2}(\Gamma) $ for some $ \epsilon > 0 $ and $\bH_0^2(\Gamma) $ into $ H^{-a}(\Gamma) $ for some $ 0 \leq a < 2 $, i.e.,
		\begin{subequations}
			\begin{align}
				& \norm{f(\bbeta_{1}) - f(\bbeta_{2})}_{H^{-2}}
				\leq C_R \norm{\bbeta_{1} - \bbeta_{2}}_{\bH^{2-\epsilon}(\Gamma)},
				\label{a1}\\
				& \norm{f(\bbeta_{1}) - f(\bbeta_{2})}_{H^{-a}}
				\leq C_R \norm{\bbeta_{1} - \bbeta_{2}}_{\bH^2(\Gamma)}. \label{a2}
			\end{align}
		\end{subequations}
		$ \eqref{a1} $ is used to pass the convergence of $ f(\bbeta) $ when the bound of approximate solutions $ \etaNn $ is obtained in $ \bH_0^2(\Gamma) $. It is a weaker Lipschitz condition than \ref{f1}.
		$ \eqref{a2} $ depicts the order of nonlinearity precisely and is useful in determining the minimal time, which is related to the approximate solution's convergence. This requirement is from the Galerkin approximation method in \cite{TW2018} to handle the nonlinear structure subproblem. However, instead of using this ``hybrid approximation scheme'' as in \cite{TW2018}, we choose the classical semigroup method in our study for the structure subproblem. Thus, we make a stronger assumption and need only one Lipschitz condition.
	\end{remark}

\subsection{Energy differential inequality}\label{energyi}
	
	For simplicity, we denote the bilinear form associated with the elastic energy of the thick structure by
	$$
		\as{\bd}{\bpsi} = \int_{\Omega_{S}} 2 \mu_{S} \bD(\bd) : \bD(\bpsi) + \lambda (\nabla \cdot \bd) (\nabla \cdot \bpsi).
	$$
	
	\begin{remark}\label{identity}
		For any two vectors $ \bd_1 $ and $ \bd_2 $, we have
		$$
			\as{\bd_1}{\bd_1 - \bd_2} = \onehalf \left( \as{\bd_1}{\bd_1} + \as{\bd_1 - \bd_2}{\bd_1 - \bd_2} - \as{\bd_2}{\bd_2} \right),
		$$
		where we used $ \bA : (\bA - \bB) = \onehalf \left( \bA : \bA + (\bA - \bB) : (\bA - \bB) - \bB : \bB \right) $ and $ a(a - b) = \onehalf ( a^2 + (a - b)^2 - b^2 ) $. Here $ \bA, \bB $ are vectors and $ a, b $ are scalar quantities.
	\end{remark}
	
	Next, we derive the energy for \eqref{NS}--\eqref{eta0}. Multiplying \eqref{NS}, \eqref{eta}, \eqref{dS} by $ \bu $, $ \pt \bbeta $, $ \pt \bd $ and integrating by parts over $ \Omega_{F}^{\bbeta}(t) $, $ \Gamma $, $ \Omega_{S} $, respectively, we then add the results and combine with   boundary conditions to obtain
	\begin{equation}\label{energy0}
		\frac{\rd}{\dt} E(t) + D (t) \leq \bC(P_{\rm in/out}(t)),
	\end{equation}
	where the energy $ E(t) $ and dissipation $ D(t) $ are denoted by
	\begin{align*}
		&
		\ad{
			E(t) : & = \onehalf \left( \norm{\bu(t)}_{\left[L^2(\Omega_{F}^{\bbeta}(t))\right]^2}^2
			+ \norm{\pt \bbeta}_{\left[L^2(\Gamma)\right]^2}^2
			+ \delta_0 \norm{\bbeta}_{\left[H_0^2(\Gamma)\right]^2}^2 \right. \\
			& \qquad \qquad \qquad \qquad \qquad \left. + \norm{\pt \bd}_{\left[L^2(\Omega_{S})\right]^2}^2
			+ \as{\bd}{\bd} \right),
		} \\
	\end{align*}
	and
	\begin{align*}
		& D(t) : = \norm{\bu}_{\left[W^{1,p}(\Omega_{F}^{\bbeta}(t))\right]^2}^p
		+ \frac{1}{\alpha} \norm{\pt \bbeta_{\tauF^{\bbeta}} - \bu_{\tauF^{\bbeta}}}_{\left[L^2(\Gamma^{\bbeta}(t))\right]^2}^2,
	\end{align*}
	respectively.
	
\subsection{The ALE mapping} \label{ALE}
	Due to the Navier-slip effect, we consider both radial (vertical) and longitudinal displacements in this study, which creates some additional issues when we pass to the limit (see e.g., Section \ref{limitproblem}). Thus, we follow the procedure in \cite[Section 3.2]{MC2016JDE} to define our ALE mapping. First, let the corresponding deformation of the elastic boundary be denoted by $ \bvphi^{\bbeta} $, i.e.,
	$$
		\bvphi ^{\bbeta}(t, z) = \textbf{\textrm{id}} + \bbeta(t,z), \quad (t,z) \in [0,T] \times \Gamma.
	$$
	Then we introduce a family of ALE mappings $ \bA_{\bbeta} $ parameterized by $ \bbeta $:
	$$
		\ad{
			\bA_{\bbeta}(t) : \Omega_{F} & \rightarrow \Omega_{F}^{\bbeta}(t), \\
			(z,r) & \mapsto (x,y) = \bA_{\bbeta}(t)(z,r),
		}
	$$
	which is defined for each $ \bbeta $ as a harmonic extension of deformation $ \bvphi^{\bbeta} $. This means that $ \bA_{\bbeta}(t) $ is the solution of the following boundary value problem on the reference domain $ \Omega_{F} $:
	\begin{align*}
		\Delta \bA_{\bbeta}(t, \cdot) & = 0 \text{ in } \Omega_{F}, \\
		\rv{\bA_{\bbeta}(t)}_{\Gamma} & = \bvphi^{\bbeta}(t, \cdot), \\
		\rv{\bA_{\bbeta}(t)}_{\Sigma} & = \textbf{\textrm{id}},
	\end{align*}
	where $ \Sigma = \partial \Omega_{F} \backslash \Gamma $ denotes the rigid part of the boundary.
	The Jacobian determinant of ALE mapping $ \bA_{\bbeta} $ is defined by
	$$
		\Jeta = \det \nabla \bA_{\bbeta} (t),
	$$
	and the ALE velocity is
	$$
		\bw^{\bbeta} = \frac{\rd}{\dt} \bA_{\bbeta}.
	$$
	
	From \cite{MC2016JDE}, we know that if the compatibility conditions \eqref{u0}, \eqref{phi0}, \eqref{eta0} hold, then it can be deduced that there exists a $ T' > 0 $ such that
	\begin{equation}
		\Jeta \geq C > 0, \quad \text{ on } (0,T') \times \Omega_{F}.
	\end{equation}
	In addition, there exists a $ T'' > 0 $ such that for $ t \in [0,T''] $, the ALE mapping $ \bA_{\bbeta} (t) $ is an injection.
	
	We note here that both conditions $ \Jeta > 0 $ and $ \bA_{\bbeta} $ is injective are to ensure that the fluid domain will not degenerate for some large time. Let $ T^* = \min \left\{ T', T'' \right\} $, then this new time determines the maximal existence time interval for the weak solution. In this sense, our weak solution exists for a maximal time $ T^* $, at which either $ \Jeta = 0 $ or $ \bA_{\bbeta} $ is no longer injective (see e.g., \cite[Fig. 3]{MC2016JDE}).
		
\subsection{Transformation settings} \label{transformation}
	In order to define our associated weak solution in the fixed domain $ \Omega_{F} $, we map functions in moving domain $ \Omega_{F}^{\bbeta}(t) $ onto the reference domain $ \Omega_{F} $ by the ALE mapping given above in Section \ref{ALE}. For functions depending on $ \bbeta $, we denote it by a superscript $ \bbeta $. Specifically, for a function $ \bm{g} $ defined on $ \Omega_{F}^{\bbeta}(t) $, whether it is a scalar or a vector, we denote it by
	$$
		\bm{g}^{\bbeta} (t,z,r) = \bm{g} (t, \bA_{\bbeta}(t)(z,r)).
	$$
	Then the gradient and divergence are given by
	$$
		\nabla^{\bbeta} \bm{g}^{\bbeta} : = \left( \nabla \bm{g} \right)^{\bbeta} = \nabla \bm{g}^{\bbeta} \left( \nabla \bA_{\bbeta} \right)^{-1},  \quad
		\nabla^{\bbeta} \cdot \bm{g}^{\bbeta} = \tr\left( \nabla^{\bbeta} \bm{g}^{\bbeta} \right)
	$$
	and the symmetric gradient is denoted by
	\begin{equation*}
		\ad{
			\bD^{\bbeta} (\bm{g}^{\bbeta}) : & = \onehalf \left( \nabla^{\bbeta} \bm{g}^{\bbeta} + \left( \nabla^{\bbeta} \bm{g}^{\bbeta} \right)^T \right) \\
			& = \onehalf \left( \left( \nabla \bm{g} \right)^{\bbeta} + \left( \left( \nabla \bm{g} \right)^{\bbeta} \right)^T \right) \\
			& = \onehalf \left( \nabla \bm{g}^{\bbeta} + \left( \nabla \bm{g}^{\bbeta} \right)^T \right) \left( \nabla \bA_{\bbeta} \right)^{-1} \\
			& = \bD (\bm{g}^{\bbeta})\left( \nabla \bA_{\bbeta} \right)^{-1}.
		}
	\end{equation*}
	Moreover, the ALE derivative on the fixed reference domain $ \Omega_{F} $ is defined by:
	\begin{equation*}
		\rv{\pt \bm{g}}_{\Omega_{F}} = \pt \bm{g} + \left( \bw^{\bbeta} \cdot \nabla \right) \bm{g}.
	\end{equation*}
	Consequently, we rewrite the Navier-Stokes equation in the ALE formulation as:
	\begin{align}\label{NSALE}
		\rv{\pt \bu}_{\Omega_{F}} + \left( \left( \bu - \bw^{\bbeta} \right) \cdot \nabla \right) \bu = \nabla \cdot \bsigma, \quad \text{ in } \Omega_{F}^{\bbeta}(t),
	\end{align}
	and
	\begin{align} \label{gradequiv}
		\nabla \bm{g} = \nabla^{\bbeta} \bm{g}^{\bbeta},
	\end{align}
	where $ \rv{\pt \bu}_{\Omega_{F}} $ and $ \bw^{\bbeta} $ are composed with $ \left( \bA_{\bbeta}(t) \right)^{-1} $ and we find that transformed divergence-free condition is
	$$
		\nabla^{\bbeta} \cdot \bu^{\bbeta} = 0.
	$$
	
	Under the above transformation with ALE mapping, we give the space settings related to weak solutions of problem \eqref{NS}--\eqref{eta0} in the next section.

\subsection{Space settings} \label{spacesetting}
	We denote the spaces $ \left[W^{k,p}\right]^2 $, $ \left[H^{p}\right]^2 $ and $ \left[L^{p}\right]^2 $ by $ \bW^{k,p} $, $ \bH^{p} $ and $ \bL^{p} $ respectively. Now, we describe the functional spaces associated with the weak solutions of problem \eqref{NS}--\eqref{eta0}.
	
	Motivated by the energy inequality \eqref{energy0} and the boundary conditions, we denote four spaces of fluid velocity, ``improved'' fluid velocity, thin structure displacement and thick structure displacement by
	\begin{equation*}
		V_{F}^{\bbeta} = \left\{ \bu^{\bbeta} \in \left[C^{1}(\overline{\Omega})\right]^2 :
		\ad{
			& \nabla^{\bbeta} \cdot \bu^{\bbeta} = 0, \bu^{\bbeta} \cdot \btau^{\bbeta} = 0, \text{ on } \Gamma_{\rm in/out}, \\
			& \bu^{\bbeta} \cdot \nuF^{\bbeta} = 0,\text{ on } \Gamma_{\rm b}
		} \right\},
	\end{equation*}
	\begin{equation*}
		\V_{F}^{\bbeta} = \overline{V_{F}^{\bbeta}}^{\bW^{1,p}(\Omega_{F})},
	\end{equation*}
	\begin{equation*}
		\V_{W} = \bH_{0}^{2}(\Gamma),
	\end{equation*}
	and
	\begin{equation*}
		\V_{S} = \left\{ \bd \in \bH^{1}(\Omega_{S}) : \bd \cdot \btau_{S} = 0, \text{ on } \Gamma, \bd = \bm{0}, \text{ on } \Gamma_{\rm in/out}^{S} \right\},
	\end{equation*}
	respectively.
	Subsequently, for $ 0 < T \leq T^* $, the associated evolution spaces for fluid, thin structure and thick structure are written as:
	\begin{equation*}
		\W_{F}^{\bbeta} = L^{\infty}(0,T; \bL^{2}(\Omega_{F})) \cap L^{p}(0,T; \V_{F}^{\bbeta}),
	\end{equation*}
	\begin{equation*}
		\W_{W} = W^{1, \infty}(0,T; \bL^{2}(\Gamma)) \cap L^{2}(0,T; \V_{W}),
	\end{equation*}
	\begin{equation*}
		\W_{S} = W^{1, \infty}(0,T; \bL^{2}(\Omega_{S})) \cap L^{2}(0,T; \V_{S}).
	\end{equation*}
	Including the \textbf{slip boundary condition}, we have the following solution space:
	\begin{equation*}
		\W^{\bbeta} = \left\{ (\bu,\bbeta,\bd) \in \W_{F}^{\bbeta} \times \W_{W} \times \W_{S} : \bu \cdot \nuF^{\bbeta} = \pt \bbeta \cdot \nuF^{\bbeta}, \bd \cdot \nuS = \bbeta \right\}.
	\end{equation*}
	The corresponding test space is denoted by
	\begin{equation} \label{testfunctionspace}
		\Q^{\bbeta} = \left\{
		\ad{
			& (\bq,\bphi,\bpsi) \in C_{c}^{1}([0,T), \V_{F}^{\bbeta} \times \V_{W} \times \V_{S}): \\
			& \qquad \bq^{\bbeta} \cdot \nuF^{\bbeta} = \bphi \cdot \nuF^{\bbeta}, \bphi = \bpsi, \text{ on }\Gamma
		} \right\}.
	\end{equation}
	
\subsection{Weak solution and main result} \label{weakmain}
	We transform problem \eqref{NS}--\eqref{eta0} by ALE mapping, so that the remainder analysis are on the reference domain $ \Omega_{F} $. To establish the definition of weak solution, we first consider the transformed Navier-Stokes equation \eqref{NSALE}. Multiplying \eqref{NSALE} by $ \bq $ and integrating it over $ \Omega_{F}^{\bbeta}(t) $, we have
	\begin{equation}\label{intptu}
		\intt \int_{\Omega_{F}^{\bbeta}(t)} \rv{\pt \bu}_{\Omega_{F}} \cdot \bq
		 + \intt \int_{\Omega_{F}^{\bbeta}(t)} \left( \left( \bu - \bw^{\bbeta} \right) \cdot \nabla \right) \bu \cdot \bq
		 = \intt \int_{\Omega_{F}^{\bbeta}(t)} \nabla \cdot \bsigma \cdot \bq,
	\end{equation}
	where we dropped the superscript $ \bbeta $ in $ \bu^{\bbeta} $ for easier reading.
	
	Integrating the second term on the left-hand side of \eqref{intptu} by parts, we get
	\begin{align*}
		& \quad \intt \int_{\Omega_{F}^{\bbeta}(t)} \left( \left( \bu - \bw^{\bbeta} \right) \cdot \nabla \right) \bu \cdot \bq \\
		& = \onehalf \intt \int_{\Omega_{F}^{\bbeta}(t)} \left( \left( \bu - \bw^{\bbeta} \right) \cdot \nabla \right) \bu \cdot \bq - \onehalf\intt \int_{\Omega_{F}^{\bbeta}(t)} \left( \left( \bu - \bw^{\bbeta} \right) \cdot \nabla \right) \bq \cdot \bu \\
		& \quad + \intt \int_{\Omega_{F}^{\bbeta}(t)} \left( \nabla \cdot \bw^{\bbeta} \right) \bu \cdot \bq + \intt \int_{\Gamma^{\bbeta}(t)} \left( \bu - \bw^{\bbeta} \right) \cdot \nuF^{\bbeta} \left( \bu \cdot \bq \right).
	\end{align*}
	For the term on the right-hand side of \eqref{intptu}, it follows from the divergence theorem that
	\begin{align*}
		- \intt \int_{\Omega_{F}^{\bbeta}(t)} \nabla \cdot \bsigma \cdot \bq = 2 \intt \int_{\Omega_{F}^{\bbeta}(t)} \bbS ( \bD(\bu) ) : \bD(\bq) - \intt \int_{\partial \Omega_{F}^{\bbeta}(t)} \bsigma \nuF^{\bbeta} \cdot \bq.
	\end{align*}
	Due to the slip effect, which leads to existence of non-zero tangential and normal component of velocity at interface boundary, we expand last term above as
	\begin{align*}
		& \intt \int_{\partial \Omega_{F}^{\bbeta}(t)} \bsigma \nuF^{\bbeta} \cdot \bq \\
		& \quad = \int_{\Gamma^{\bbeta}(t)} \biggl( \left( \bsigma \nuF^{\bbeta} \cdot \nuF^{\bbeta} \right) \bq \cdot \nuF^{\bbeta} + \left( \bsigma \nuF^{\bbeta} \cdot \tauF^{\bbeta} \right) \bq \cdot \tauF^{\bbeta} \biggr)
		+ \int_{\Gamma_{\rm in/out}} \pi \bq_{\nuF} \\
		& \quad = \int_{\Gamma^{\bbeta}(t)} \biggl( \left( \bsigma \nuF^{\bbeta} \cdot \nuF^{\bbeta} \right) \bphi \cdot \nuF^{\bbeta} + \frac{1}{\alpha} \left( \pt \bbeta - \bu \right) \cdot \tauF^{\bbeta} \left(  \bq \cdot \tauF^{\bbeta}  \right) \biggr)
		+ \int_{\Gamma_{\rm in/out}} \pi \bq_{\nuF}.
	\end{align*}
	We note here that the first term on the right-hand side cancels with the same term in thin structure equation.
	By means of integration by parts, we obtain
	\begin{equation}\label{ptuOmega}
		\ad{
			& \intt \int_{\Omega_{F}^{\bbeta}(t)} \rv{\pt \bu}_{\Omega_{F}} \cdot \bq
			= \intt \int_{\Omega_{F}} \Jeta \pt \bu \cdot \bq^{\bbeta} \\
			& \quad = - \intt \int_{\Omega_{F}} \pt \Jeta \bu \cdot \bq^{\bbeta} - \intt \int_{\Omega_{F}} \Jeta \bu \cdot \pt \bq^{\bbeta}
			- \int_{\Omega_{F}} J_0 \bu_{0} \cdot \bq^{\bbeta}(0, \cdot).
		}
	\end{equation}
	Since we have (see e.g., \cite[pp. 77]{Gurtin1981})
	\begin{equation*}
		\pt \Jeta = \Jeta \nabla^{\bbeta} \cdot \bw^{\bbeta},
	\end{equation*}
	\eqref{ptuOmega} becomes
	\begin{align*}
		& \intt \int_{\Omega_{F}^{\bbeta}(t)} \rv{\pt \bu}_{\Omega_{F}} \cdot \bq \\
		& \quad = - \intt \int_{\Omega_{F}} \biggl( \Jeta \left( \nabla^{\bbeta} \cdot \bw^{\bbeta} \right) \left( \bu \cdot \bq^{\bbeta} \right)
		- \Jeta \bu \cdot \pt \bq^{\bbeta} \biggr)
		- \int_{\Omega_{F}} J_0 \bu_{0} \cdot \bq^{\bbeta}(0, \cdot).
	\end{align*}
	
	We multiply the elastic equation of $ \bbeta $ and $ \bd $ by $ \bphi $ and $ \bpsi $, respectively, and integrate by parts over $ (0,T) \times \Gamma $ and $ (0,T) \times \Omega_{S} $, respectively. Then, we add the results with \eqref{intptu} together to obtain the weak formulation.
	
	\begin{definition}[Weak solution]\label{weaksolution}
		Assume that assumptions \ref{f1} holds. Then $ (\bu, \bbeta, \bd) \in \W^{\bbeta} $ is a weak solution of \eqref{NS}--\eqref{eta0} in $ (0,T) $ for $ 0 < T \leq T^* $, if for every $ (\bq, \bphi, \bpsi) \in \Q^{\bbeta} $, the following equality holds:
		\begin{align}
			& \quad \onehalf \intt \int_{\Omega_{F}} \Jeta\bigg(
			\left( \left( \bu - \bw^{\bbeta} \right) \cdot \nabla^{\bbeta} \right) \bu \cdot \bq
			- \left( \left( \bu - \bw^{\bbeta} \right) \cdot \nabla^{\bbeta} \right) \bq \cdot \bu
			- \left( \nabla^{\bbeta} \cdot \bw^{\bbeta} \right) \cdot \bq \cdot \bu
			\bigg) \nonumber \\
			& - \intt \int_{\Omega_{F}} \Jeta \bu \cdot \pt \bq
			+ 2 \intt \int_{\Omega_{F}} \Jeta \bbS(\bD^{\bbeta}(\bu)) : \bD^{\bbeta}(\bq)
			+ \intt \inner{f(\bbeta)}{\bphi} \nonumber \\
			& + \frac{1}{\alpha} \intt \int_{\Gamma} ( \bu_{\tauF^{\bbeta}} - \pt \bbeta_{\tauF^{\bbeta}} ) \bq_{\tauF^{\bbeta}} \Jeta_{F} \dz \dt
			- \intt \int_{\Gamma} \pt \bbeta \pt \bphi
			+ \intt \inner{\LE \bbeta}{\bphi}  \label{weakformulation}  \\
			& + \frac{1}{\alpha} \intt \int_{\Gamma} ( \pt \bbeta_{\tauF^{\bbeta}}
			- \bu_{\tauF^{\bbeta}} ) \bphi_{\tauF^{\bbeta}} \Jeta_{F} \dz \dt
			- \intt \int_{\Omega_{S}} \pt \bd \cdot \pt \bpsi
			+ \intt \as{\bd}{\bpsi} \nonumber \\
			= &\intt \int_{\Gamma_{\rm in/out}} P_{\rm in/out} \bq \cdot \nuF + \int_{\Omega_{F}}J_{0}\bu_{0} \cdot \bq(0) + \int_{\Gamma} \bv_{0} \bphi(0) + \int_{\Omega_{S}} \bV_{0} \cdot \bpsi(0). \nonumber
		\end{align}
	\end{definition}
	
	Now,  we state the main result.
	\begin{theorem}[Main result]\label{mainresult}
		Let $ P_{\rm in/out} \in L_{\rm loc}^2 (0, \infty) $. Suppose that the initial data $ \bu_{0} \in \bL^{2}(\Omega_{F}) $, $ \bbeta_{0} \in \bH_{0}^{1}(\Gamma) $, $ \bv_{0} \in \bL^{2}(\Gamma) $, $ \bd_{0} \in \bH^{1}(\Omega_{S}) $ and $ \bV_{0} \in \bL^{2}(\Omega_{S}) $ satisfy the compatibility conditions \eqref{u0}, \eqref{phi0} and \eqref{eta0}, then under assumption \ref{f1}, there exist a $ T^* > 0 $ and a weak solution $ (\bu,\bbeta,\bd) $ to  \eqref{NS}--\eqref{eta0} on $ (0,T^*) $ in the sense of Definition \ref{weaksolution} such that the following energy estimate holds:
		\begin{equation}\label{energyestimate}
			\ad{
				E(t) + \int_{0}^{t} D(s) \ds
				\leq E_{0}
				+ C \left( \norm{P_{\rm in}}_{L^2(0,T)}^2 + \norm{P_{\rm out}}_{L^2(0,T)}^2 + T^* \right).
			}
		\end{equation}
		where $ E(t) $ and $ D(t) $ are
		$$
		\ad{
			E(t) & = \onehalf \left( \norm{\bu}_{\bL^{2} (\Omega_{F}^{\bbeta}(t))}^2
			+ \normg{\pt \bbeta}
			+ \delta_0 \normgh{\bbeta}
			+ \normos{\pt \bd}
			+ \as{\bd}{\bd}
			\right), \\
			D(t) & = \norm{\bu}_{\bW^{1,p}(\Omega_{F}^{\bbeta}(t))}^p
			+ \frac{1}{\alpha}\norm{ \pt \bbeta_{\tauF^{\bbeta} } - \bu_{\tauF^{\bbeta}} }^{2}_{\bL^{2}(\Gamma^{\bbeta}(t))},
		}	
		$$
		and $ E_{0} = E(0) $.
	\end{theorem}
	
	\section{Approximate solutions}
	\subsection{Operator splitting scheme}
	\label{splitting}
	In this section, the backward Euler scheme is used to define a sequence of approximate solutions of the fluid-structure interaction problem. For every fixed $ T > 0 $ and $ N \geq 1 $, we devide the interval $ [0,T] $ into $ N $ subintervals of length $ \Dt = \frac{T}{N} $ with $ 0 = t_0 < t_1 < \cdots < t_{N-1} < t_N = T $. In the subintervals of $ [0,T] $ we separate the problem into two parts by Lie operator splitting method as follows.
	
	First, we rewrite problem \eqref{NS}--\eqref{eta0} as
	\begin{equation*}
		\left\{
		\begin{aligned}
			& \frac{\mathrm{d} \bX}{\dt} && = && A \bX , \quad t \in (0, T), \\
			& \left. \bX \right|_{t = 0} && = && \bX^0.
		\end{aligned}
		\right.
	\end{equation*}
	where $ \bX = ( \bu, \bbeta, \bv, \bd, \bV)^{T} $ and $ \bX^{0} = (\bu_{0}, \bbeta_{0}, \bv_{0}, \bd_{0}, \bV_{0})^{T} $.
	
	Then, we decompose $ A = A_1 + A_2 $, where $ A_1 $ and $ A_2 $ are non-trivial, and for $ n = 0,1,\dots, N - 1, i = 1,2 $, we obtain
	$$
		\left\{
		\begin{aligned}
			& \frac{\mathrm{d} \bX_{N}^{n+\frac{i}{2}}}{\dt} && = && A_i \bX_{N}^{n+\frac{i}{2}}, \quad t \in (t_n, t_{n+1}) \\
			& \left. \bX_{N}^{n+\frac{i}{2}} \right|_{t = t_n} && = && \bX_{N}^{n+\frac{i-1}{2}}.
		\end{aligned}
		\right.
	$$
	which can be solved for the approximate vector solutions
	$$
		\bX_{N}^{n+\frac{i}{2}} = (\un{n+\frac{i}{2}}, \etan{n+\frac{i}{2}}, \vn{n+\frac{i}{2}}, \dn{n+\frac{i}{2}}, \VN{n+\frac{i}{2}})^{T},
	$$
	with $ \bX^{0} = (\bu_{0}, \bbeta_{0}, \bv_{0}, \bd_{0}, \bV_{0})^{T} $, where $ i = 1,2 $ denotes the solution of the structure and of the fluid subproblem, respectively.
	
	In the following, we write the subproblems under the time discretization while we omit the subscript $ N $ for simplicity.
	
\subsection{The structure subproblem}
	\label{structuresubproblem}
	In the structure subproblem, we notice that $ \bu $ does not change, then we denote
	$$
		\bu^{n+\onehalf} = \bu^n.
	$$
	Let
	$$
		\cH =  \V_{W} \times \V_{W} \times \V_{S} \times \V_{S}
	$$
	and
	$$
		\tilde{\cH} := \left\{ (\bphi, \bpsi)^T \in \V_{W} \times \V_{S} : \rv{\bpsi}_{\Gamma} = \bphi \right\}.
	$$
	Fixing $ \Dt $ and defining the solution of structure subproblem by $ (\bbeta^{n+\onehalf}, \bv^{n+\onehalf}, \bd^{n+\onehalf}, \bV^{n+\onehalf}) \in \cH $, we have the weak formulation of structure subproblem:
	
	For $ (\bbeta^{n}, \bd^{n})^T \in \tilde{\cH} $, find $ (\bbeta^{n+\onehalf}, \bv^{n+\onehalf}, \bd^{n+\onehalf}, \bV^{n+\onehalf})^T \in \cH $ such that
	\begin{equation}\label{weaks}
		\begin{gathered}
			\rv{\bd^{n+\onehalf}}_{\Gamma} = \bbeta^{n+\onehalf}, \quad \frac{\bbeta^{n+\onehalf} - \bbeta^n}{\Dt} = \bv^{n+\onehalf}, \quad \frac{\bd^{n+\onehalf} - \bd^n}{\Dt} = \bV^{n+\onehalf},\\
			\ad{
				\int_{\Gamma} \frac{\bv^{n+\onehalf} - \bv^n}{\Dt} \cdot \bphi
				& + \int_{\Omega_{S}} \frac{\bV^{n+\onehalf} - \bV^n}{\Dt} \cdot \bpsi + \inner{\LE \bbeta^{n+\onehalf}}{\bphi} \\
				& + \as{\bd^{n+\onehalf}}{\bpsi} + \int_{\Gamma} f(\bbeta^{n+\onehalf}) \cdot \bphi = 0,
			}
		\end{gathered}
	\end{equation}
	for all $ (\bphi, \bpsi)^T \in \tilde{\cH} $.
	
	In \eqref{weaks}, the equations in the first row are kinematic coupling conditions and the second row is the weak form. We solve $ (\bbeta^{n+\onehalf}, \bv^{n+\onehalf}, \bd^{n+\onehalf}, \bV^{n+\onehalf})^T \in \cH $ with $ (\bbeta^{n}, \bd^{n})^T \in \tilde{\cH} $ under the invariance of $ \bu^{n+\onehalf} = \bu^n $.
	
	\begin{lemma}\label{structurepro}
		For a fixed $ \Dt >0 $, there exists a unique weak solution $ ( \bbeta^{n+\onehalf}, \bv^{n+\onehalf}, \bd^{n+\onehalf}, \bV^{n+\onehalf} ) ^T $$\in \cH $ to subproblem \eqref{weaks} with $ (\bbeta^{n}, \bd^{n})^T \in \tilde{\cH} $.
	\end{lemma}

	As stated in Section \ref{methodology}, Lax-Milgram theorem does not hold any more. So we rewrite the structure subproblem as the form of continuous ordinary differential system in a subinterval $ (t_n, t_{n+1}) $ and apply the Theorem 6.1.2 in \cite{Pazy1983} to complete the proof as follows.
		
	First, define two linear self-adjoint operator as
	$$
		\inner{\LA \bbeta}{\bv} = \inner{\bS(\bd) \nuS}{\bv}_{\Gamma},
	$$
	and
	$$
		\inner{\LB \bd}{\bd'} = \as{\bd}{\bd'}
		- \inner{\bS(\bd) \nuS}{\pt \bd'}_{\Gamma}.
	$$
	Denoting $ \bU(t) $ by
	$$
		\displaystyle \bU(t) = (\bbeta, \bv, \bd, \bV)^{T},
	$$
	we note that for $ \bU, \tilde{\bU} = (\tilde{\bbeta}, \tilde{\bv}, \tilde{\bd}, \tilde{\bV})^{T} \in \cH $,
	\begin{equation*}
		\inner{\bU}{\tilde{\bU}}_{\cH}
		= \inner{\LE \bbeta}{\tilde{\bbeta}}
		+ \inner{\bv}{\tilde{\bv}}
		+ \as{\bd}{\tilde{\bd}}
		+ \inner{\bV}{\tilde{\bV}}.
	\end{equation*}

	Then the equivalent continuous structure sub-problem in $ (t_n, t_{n+1}) $ is:
	for $ \bU(t_{n}) = (\bbeta^{n}, \bv^{n}, \bd^{n}, \bV^{n})^{T} \in \cH $, find $ \bU = \bU(t_{n+\onehalf}) $, such that
	\begin{equation}\label{subf}
		\frac{\rd}{\dt} \bU + \cA \bU = \cF(\bU)
	\end{equation}
	with
	\begin{equation*}
		\cA = \left(
		\begin{array}{cccc}
			0 & -1 & 0 & 0 \\
			\LE + \LA & 0 & 0 & 0 \\
			0 & 0 & 0 & -1 \\
			0 & 0 & \LB & 0
		\end{array}
		\right),\quad
		\cF(\bU) = \left(
		\begin{array}{c}
		0 \\
		- f(\bbeta) \\
		0 \\
		0
		\end{array}
		\right),
	\end{equation*}
	and
	$$
		\cD(\cA) = \left\{ \bU \in \cH \left\vert \LE \bbeta + \LA \bbeta \in \V_{W}, \LB \bd \in \V_{S} \right. \right\}.
	$$
%
%
%

	In the sequel, we will prove that \eqref{subf} admits a unique solution $ \bU(t_n) \in \cH $, i.e., Lemma \ref{structurepro} holds.
	
	\begin{proof}[Proof of Lemma \ref{structurepro}]	
	\textbf{Step 1}. $ \inner{\cA \bU}{\bU}_{\cH} \geq 0 $.
	
	From the definition of $ \cA $, we have
	$$
		\cA \bU = \left( - \bv, \LE \bbeta + \LA \bbeta, -\bV, \LB \bd \right)^{T}.
	$$
	Then, we get
	\begin{equation*}
		\ad{
		\inner{\cA \bU}{\bU}_{\cH}
		& = \inner{ - \LE \bv}{\bbeta}
			+ \inner{\LE \bbeta}{\bv}
			+ \inner{\LA \bbeta}{\bv}
			+ \as{-\bV}{\bd}
			+ \inner{\LB \bd}{\bV} \\
		& = \inner{\bS\nuS}{\bv}_{\Gamma}
			- \as{\bV}{\bd}
			+ \as{\bd}{\bV}
			+ \inner{\bS\nuS}{\bv}_{\Gamma} \\
		& \geq 0,
		}
	\end{equation*}
	where we used the boundary condition $ \pt \bd = \bV = \bv $ on $ \Gamma $.
	
	\textbf{Step 2}. $ R(\mathbb{I} + \cA) = \cH $.
	
	First, we prove that $ \mathbb{I} + \cA $ is surjective, i.e., for every $ G = (g_{1}, g_{2}, g_{3}, g_{4})^{T} \in \cH $, there exists $ V = (v_{1}, v_{2}, v_{3}, v_{4})^{T} \in \cD(\cA) $ such that
	\begin{equation}\label{surj}
		(\mathbb{I} + \cA)V = G,
	\end{equation}
	that is,
	\begin{subequations}\label{surjective1}
		\begin{numcases}{}
			v_{1} - v_{2} = g_{1}, \label{surjective1a}\\
			v_{2} + \LE v_{1} + \LA v_{1} = g_{2}, \label{surjective1b}\\
			v_{3} - v_{4} = g_{3}, \label{surjective1c}\\
			v_{4} + \LB v_{3} = g_{4}.\label{surjective1d}
		\end{numcases}
	\end{subequations}
	Adding the first two equations and last two equations in \eqref{surjective1} respectively, we obtain
	\begin{equation}\label{surjective2}
	\left\{
		\ad{
			& v_{1} + \LE v_{1} + \LA v_{1} = g_{1} + g_{2}, \\
			& v_{3} + \LB v_{3} = g_{3} + g_{4}.
		}
	\right.
	\end{equation}
	Multiplying \eqref{surjective2} by $ (\tilde{v}_{1}, \tilde{v}_{3})^{T} \in \tilde{\cH} $ and integrating equations over $ \Gamma, \Omega_{S} $ respectively, we find that
	\begin{equation}\label{surjective3}
	\left\{
	\ad{
		& \int_{\Gamma} v_{1} \cdot \tilde{v}_{1} + \int_{\Gamma}\LE v_{1} \cdot \tilde{v}_{1} + \int_{\Gamma} \LA v_{1} \cdot \tilde{v}_{1}
		= \int_{\Gamma} g_{1} \cdot \tilde{v}_{1} + \int_{\Gamma} g_{2} \cdot \tilde{v}_{1}, \\
		& \int_{\Omega_{S}} v_{3} \cdot \tilde{v}_{3} + \int_{\Omega_{S}} \LB v_{3} \cdot \tilde{v}_{3}
		= \int_{\Omega_{S}} g_{3} \cdot \tilde{v}_{3} + \int_{\Omega_{S}} g_{4} \cdot \tilde{v}_{3}.
	}
	\right.
	\end{equation}
	Since $ \inner{\LA v_{1}}{\tilde{v}_{1}} + \inner{\LB v_{3}}{\tilde{v
	}_{3}} = \as{v_{3}}{\tilde{v}_{3}} $ with $ v_{1} = v_{3} $ and $ \tilde{v}_{1} = \tilde{v}_{3} $ on $ \Gamma $, we have the following variational formulation:
	\begin{equation*}
		B\left((v_{1},v_{3})^{T}, (\tilde{v}_{1},\tilde{v}_{3})^{T}\right)
		= \tilde{B}\left((\tilde{v}_{1},\tilde{v}_{3})^{T}\right),
		\quad \forall\ (\tilde{v}_{1}, \tilde{v}_{3})^{T} \in \tilde{\cH},
	\end{equation*}
	where
	\begin{equation*}
		B\left((v_{1},v_{3})^{T}, (\tilde{v}_{1},\tilde{v}_{3})^{T}\right)
		= \int_{\Gamma} v_{1} \cdot \tilde{v}_{1}
		+ \int_{\Gamma} \LE v_{1} \cdot \tilde{v}_{1}
		+ \int_{\Omega_{S}} v_{3} \cdot \tilde{v}_{3}
		+ \as{v_{3}}{\tilde{v}_{3}},
	\end{equation*}
	and
	\begin{equation*}
		\tilde{B}\left((\tilde{v}_{1},\tilde{v}_{3})^{T}\right)
		= \int_{\Gamma} g_{1} \cdot \tilde{v}_{1} + \int_{\Gamma} g_{2} \cdot \tilde{v}_{1} + \int_{\Omega_{S}} g_{3} \cdot \tilde{v}_{3} + \int_{\Omega_{S}} g_{4} \cdot \tilde{v}_{3}.
	\end{equation*}

	Now we introduce the norm of the Hilbert space $ \tilde{\cH} $ as 
	$$
		\norm{(v_{1},v_{3})}^{2}_{\tilde{\cH}}
		= \norm{v_{1}}^{2}_{\bL^{2}(\Gamma)}
		+ \norm{v_{1}}^{2}_{\bH^{2}(\Gamma)}
		+ \norm{v_{3}}^{2}_{\bL^{2}(\Omega_{S})}
		+ \as{v_{3}}{v_{3}}.
	$$
	Then it can be deduced that the bilinear $ B(\cdot , \cdot) $ and the functional $ \tilde{B}(\cdot) $ are bounded. Furthermore, it follows that there exists a positive constant $ \delta_1 $ such that
	$$
	\ad{
		B\left((v_{1},v_{3})^{T}, (v_{1},v_{3})^{T}\right)
		& = \int_{\Gamma} v_{1} \cdot v_{1}
			+ \int_{\Gamma} \LE v_{1} \cdot v_{1}
			+ \int_{\Omega_{S}} v_{3} \cdot v_{3}
			+ \as{v_{3}}{v_{3}} \\
		& \geq \delta_1 \norm{(v_{1},v_{3})}^{2}_{\tilde{\cH}},
	}
	$$
	which means that $ B(\cdot , \cdot) $ is coercive.
	
	By applying the Lax-Milgram Lemma \cite{Pazy1983}, system \eqref{surjective3} has a unique solution $ (v_{1},v_{3})^{T} \in \tilde{\cH} $. In \eqref{surjective1a} and \eqref{surjective1c}, we see that
	$$
		v_{2} \in \V_{W}, \quad v_{4} \in \V_{S}.
	$$
	Then it follows from \eqref{surjective1b} and \eqref{surjective1d} that
	\begin{gather*}
		\LE v_{1} + \LA v_{1} = g_{2} - v_{2} \in \V_{W}, \\
		\LB v_{3} = g_{4} - v_{4} \in \V_{S}.
	\end{gather*}
	As a consequence, there exists a unique solution $ \bU \in \cD(\cA) $  such that \eqref{surj} is satisfied, which means that $ \mathbb{I} + \cA $ is surjective.
	
	\textbf{Step 1} and \textbf{step 2} tell us that $ \cA $ is a maximal monotone operator. Then by the Lumer-Phillips theorem \cite[Theorem 1.4.3]{Pazy1983}, $ \cA $ generates a semigroup of contractions in $ \cH $.
	
	\textbf{Step 3}. $ \cF $ is locally Lipschitz in $ \cH $.
	
	In this step, we show that $ f $ is locally Lipschitz from $ \bH^{2} $ into $ H^{2} $ and this is true due to the assumption  \ref{f1}.
	
	By \cite[Theorem 6.1.2]{Pazy1983}, we prove the existence and uniqueness of the solutions.
	\end{proof}
	
	For readability, we will replace the superscript $ \bbeta^{n} $ of fluid variables with $ n $ in the sequel.
	
\subsection{The fluid subproblem}
	\label{fluidsubproblem}
	Before investigating the fluid subproblem, we introduce the translation in time by $ \Dt = T/N $ of a function $ f $, denoted by $ \TN $, as
	\begin{align*}
		\TN f(t, \cdot) = f(t - \Dt, \cdot).
	\end{align*}
	
	Suppose that we have solved problem \eqref{weaks}. Since $ \bbeta $ and $ \bd $ do not change in the fluid subproblem, we set $ \bbeta^{n+1} = \bbeta^{n+\onehalf} $, $ \bd^{n+1} = \bd^{n+\onehalf} $. In this work, our idea is to update the fluid domain when we are in structure subproblem, that is, in a fixed time subinterval $ (n\Dt, (n+1)\Dt) $, fluid domain is $ \Omega_{F}^{n}(t) $ which is parameterized by $ \bbeta^{n} $, not by $ \bbeta^{n+1} $. Thus, the gradient of fluid velocity $ \nabla^{\bbeta^{n+1}} $ is not accurate and it need to be modified. Due to the equivalent time step $ \Dt $, we define the discrete time shift displacement $ \tilde{\bbeta}^{n+1} $ as
	\begin{align*}
		\tilde{\bbeta}^{n+1} = \TN \bbeta^{n+1} = \bbeta^{n+1} (t-\Dt, \cdot).
	\end{align*}
	Then the gradient of fluid velocity should be $ \nabla^{\tilde{\bbeta}^{n+1}} $, and so is the symmetric gradient $ \bD^{\tilde{\bbeta}^{n+1}} $ in fluid subproblem.
	
	Based on the updated position $ \bbeta^{n+1} $ of thin elastic structure, we define the ALE mapping $ \bA^{n+1} $ as the harmonic extension of $ \bbeta^{n+1} $, i.e., $ \bA^{n+1} = \textbf{id} + \bB^{n+1} $, where $ \bB^{n+1} $ is the solution of the elliptic boundary value problem:
	\begin{align*}
		\Delta \bB^{n+1} & = 0 \quad \text{ in } \Omega_{F}, \\
		\rv{\bB^{n+1}}_{\Gamma} & = \bbeta^{n+1}, \\
		\rv{\bB^{n+1}}_{\Sigma} & = 0.
	\end{align*}
	Then we have the discrete ALE velocity
	$$
		\bw^{n+1} = \frac{\bA^{n+1} - \bA^n}{\Dt}, \quad \rv{\bw^{n+1}}_{\Gamma} = \frac{\bbeta^{n+1} - \bbeta^n}{\Dt},
	$$
	and the Jacobian determinant of the ALE mapping
	$$
		J^{n+1} = \det \nabla \bA^{n+1}, \quad \bA^{n+1} : \Omega_{F} \rightarrow \Omega_{F}^{\bbeta^{n+1}}.
	$$
	
	Let
	$$
		\sH = \left\{ (\bq, \bphi)^T \in \Ve_{F} \times \bL^2(\Gamma) : \bq^{n} \cdot \nuF^{n} = \bphi \cdot \nuF^{n} \text{ on }\Gamma \right\}.
	$$
	Given $ (\bu^{n+\onehalf} , \bv^{n+\onehalf}) \in \sH $, we can find $ (\bu^{n+1} , \bv^{n+1}) \in \sH $, such that
	\begin{equation}\label{fluidsubpro}
	\sA (\bu^{n+1}, \bv^{n+1}) = \sF (\bu^{n+\onehalf}, \bv^{n+\onehalf}),
	\end{equation}
	where
	\begin{align*}
		& \quad\  \inner{\sA (\bu^{n+1}, \bv^{n+1})}{(\bq, \bphi)^{T}} \\
		& = \int_{\Omega_{F}} J^n \bu^{n+1} \cdot \bq + \onehalf \int_{\Omega_{F}} (J^{n+1} - J^{n}) \bu^{n+1} \cdot \bq \\
		& \quad + \frac{\Dt}{2}\int_{\Omega_{F}} J^{n} \biggl( \big( (\bu^{n} - \bw^{n+1}) \cdot \nabla^{\tetan{n+1}} \big) \bu^{n+1} \cdot \bq
		\bigg.   \bigg. - \big( (\bu^{n} - \bw^{n+1}) \cdot \nabla^{\tetan{n+1}} \big) \bq \cdot \bu^{n+1} \biggr) \\
		& \quad + 2 \Dt \int_{\Omega_{F}} J^n \bbS(\bD^{\tetan{n+1}}(\bu^{n+1})) : \bD^{\tetan{n+1}}(\bq) \\
		& \quad + \frac{\Dt}{\alpha} \int_{\Gamma} (\bu_{\tauF^{n+1}}^{n+1} - \bv_{\tauF^{n+1}}^{n+1}) \bq_{\tauF^{n+1}} J_{F}^{n+1} + \int_{\Gamma} \bv^{n+1} \cdot \bphi \\
		& \quad + \frac{\Dt}{\alpha} \int_{\Gamma} (\bv_{\tauF^{n+1}}^{n+1} - \bu_{\tauF^{n+1}}^{n+1}) \bphi_{\tauF^{n+1}} J_{F}^{n+1},
	\end{align*}
	and
	\begin{align*}
		& \quad\  \inner{ \sF (\bu^{n+\onehalf}, \bv^{n+\onehalf}) }{(\bq , \bphi)^{T}} \\
		& = \int_{\Omega_{F}} J^{n} \bu^{n+\onehalf} \cdot \bq + \int_{\Gamma} \bv^{n+\onehalf} \cdot \bphi + \Dt P_{\rm in/out}^{n} \int_{\Gamma_{\rm in/out}} \bq \cdot \nuF,
	\end{align*}
	for $ (\bq, \bphi)^T \in \sH $ and $ P_{\rm in/out}^{n} = \frac{1}{\Dt} \int_{n\Dt}^{(n+1)\Dt} P_{\rm in/out} $.
	The corresponding weak form is
	\begin{equation}\label{weakf}
		\inner{\sA (\bu^{n+1}, \bv^{n+1})}{(\bq, \bphi)^{T}} = \inner{ \sF (\bu^{n+\onehalf}, \bv^{n+\onehalf}) }{(\bq , \bphi)^{T}}.
	\end{equation}
	
	\begin{remark}
		Notice that the solution $ \bu^{n+1} $ and the test function $ \bq $ are all defined in $ \Ve_{F} $ in which the functions are all in domain $ \Omega_{F}^n $. This is due to our idea of iteration is to update the fluid domain when we are in structure subproblem, that is, in a fixed time subinterval $ (n\Dt, (n+1)\Dt) $, fluid domain is $ \Omega_{F}^{n}(t) $ which is parameterized by $ \eta^{n} $, not by $ \eta^{n+1} $. This choice will not affect the limit of solutions as $ N \rightarrow \infty $, see also \cite{MC2019JDE}.
	\end{remark}
	
	\begin{lemma}\label{fluidpro}
		For a fixed $ \Dt >0 $, there exists a unique weak solution $ ( \bu^{n+1}, \bv^{n+1} ) ^T \in \sH $ to subproblem \eqref{fluidsubpro} with $ (\bu^{n+\onehalf}, \bv^{n+\onehalf})^T \in \sH $.
	\end{lemma}
	In this paper, inspired by the $ p $-structure of $ \mu_{F}(\bD(\bu)) $, we solve the fluid  subproblem by the Browder-Minty theorem \cite[Theorem 9.14--1]{Ciarlet2013}.
	
	First, we state the Browder-Minty theorem.
	\begin{proposition}[Browder-Minty theorem \cite{Browder1963,Ciarlet2013,Minty1963}]\label{BMT}
		Let $ V $ be a real separable reflexive Banach space and let $ A: V \rightarrow V' $ be a coercive and hemicontinuous monotone operator. Then $ A $ is surjective, i.e., given any $ f \in V' $, there exists $ u $ such that
		$$
			u \in V \quad \text{and} \quad A(u) = f.
		$$
		If $ A $ is strictly monotone, then $ A $ is also injective.
	\end{proposition}
	
	\begin{proof}[Proof of Lemma \ref{fluidpro}]
		\textbf{Hemicontinuity.} From the definition of $ \sH $, we know that $ \sA : \sH \rightarrow \sH' $ is a bounded operator due to the H\"{o}lder' s inequality and \eqref{Boundedness}. Since all parts except the Non-Newtonian term are linear, the boundedness implies the continuity. It remains to be shown that the nonlinear part is hemicontinuous. For this we define the operator $ \sN : \Ve_{F} \rightarrow (\Ve_{F})'$ such that
		$$
			\inner{\sN(\bu^{n+1})}{\bq} = \int_{\Omega_{F}} J^n \bbS(\bD^{\tetan{n+1}}(\bu^{n+1})) : \bD^{\tetan{n+1}}(\bq), \ \forall\ \bq \in \Ve_{F}.
		$$
		Then the mapping
		$$
			s \in \mathbb{R} \rightarrow \inner{\sN(\bu^{n+1} + s \tilde{\bu})}{\bq} \in \mathbb{R}, \ \forall\ \tilde{\bu} \in \Ve_{F}
		$$
		is continuous with repect to $ s $. Therefore, $ \sN : \Ve_{F} \rightarrow (\Ve_{F})'$ is hemicontinuous \cite[Proof of Theorem 9.14-2]{Ciarlet2013} and $ \sA $ is hemicontinuous.
		\\\textbf{Coercivity.} Taking $ \bq = \bu^{n+1} $ and $ \bpsi = \bv^{n+1} $, we find that there is a $ \delta_2 $, such that
		\begin{align*}
			& \quad\  \inner{\sA (\bu^{n+1}, \bv^{n+1})}{(\bu^{n+1}, \bv^{n+1})^{T}} \\
			& = \int_{\Omega_{F}} J^{n} \abs{\bu^{n+1}}^2 + \onehalf \int_{\Omega_{F}}(J^{n+1} - J^{n}) \abs{u^{n+1}}^2 \\
			& \quad + 2 \Dt \int_{\Omega_{F}} J^n \bbS(\bD^{\tetan{n+1}}(\bu^{n+1})) : \bD^{\tetan{n+1}}(\bu^{n+1})
			\\
			& \quad + \frac{\Dt}{\alpha} \int_{\Gamma} \abs{\bu_{\tauF^{n+1}}^{n+1} - \bv_{\tauF^{n+1}}^{n+1}}^2 J_{F}^{n+1} + \int_{\Gamma} \abs{\bv^{n+1}}^2 \\
			& \geq \delta_2 \left( \int_{\Omega_{F}} J^{n} \abs{\bu^{n+1}}^2 + \int_{\Omega_{F}} J^n \abs{\bD^{\tetan{n+1}}( \bu^{n+1} )}^{p} + \normg{\bv^{n+1}} \right),
		\end{align*}
		where we used the property \eqref{Coercivity} to deal with the $ p $-structure of $ \mu_{F} $ and taken $ \alpha $ suitably small such that $ \frac{\Dt}{\alpha} \int_{\Gamma} \abs{\bu_{\tauF^{n+1}}^{n+1} - \bv_{\tauF^{n+1}}^{n+1}}^2 J_{F}^{n+1} - 2 \Dt \int_{\Omega_{F}} J^n \kappa_{2} \geq 0 $. Then the coercivity of $ \sA $ is verified.
		\\\textbf{Strict monotone.} For two different variables pairs $ (\bu_{1}^{n+1} , \bv_{1}^{n+1})^{T} $ and $ (\bu_{2}^{n+1} , \bv_{2}^{n+1})^{T} $, it can be shown from \eqref{Monotonicity} that
		\begin{align*}
			& \quad \  \inner{\sA (\bu_{1}^{n+1}, \bv_{1}^{n+1}) - \sA (\bu_{2}^{n+1}, \bv_{2}^{n+1})}{(\bu_{1}^{n+1} - \bu_{2}^{n+1}, \bv_{1}^{n+1} - \bv_{2}^{n+1})^{T}} \\
			& = \int_{\Omega_{F}} J^{n} \abs{\bu_{1}^{n+1} - \bu_{2}^{n+1}}^2 + \onehalf \int_{\Omega_{F}} (J^{n+1} - J^{n}) \abs{\bu_{1}^{n+1} - \bu_{2}^{n+1}}^2 \\
			& \quad + 2\Dt \int_{\Omega_{F}} J^n \left( \bbS(\bD_{1}) - \bbS(\bD_{2}) \right) : \left( \bD_{1} - \bD_{2} \right) \\
			& \quad + \int_{\Gamma} \abs{\bv_{1}^{n+1} - \bv_{2}^{n+1}}^{2}
			+ \frac{\Dt}{\alpha} \int_{\Gamma} \abs{\left( \bu_{1\tauF^{n+1}}^{n+1} - \bu_{2\tauF^{n+1}}^{n+1} \right) - \left( \bv_{1\tauF^{n+1}}^{n+1} - \bv_{2\tauF^{n+1}}^{n+1} \right)}^{2} J_{F}^{n+1} \\
			& > 0,
		\end{align*}
		where $ \bD_{1} = \bD^{\tetan{n+1}}(\bu_{1}^{n+1}) $ and $ \bD_{2} = \bD^{\tetan{n+1}}(\bu_{2}^{n+1}) $. Thus $ \sA $ is strictly monotone from the definition.
		
		Therefore, the proof is complete by means of Proposition \ref{BMT}.
	\end{proof}
	
\subsection{Uniform energy estimates}
	\label{unifromestimates}
	According to the decomposition of problem \eqref{NS}--\eqref{eta0}, we define the semi-discrete kinematic energy, elastic energy, total energy and dissipation in a time subinterval $ (n\Dt, (n+1)\Dt] $ by
	\begin{align}
		& E_{kin}^{n+\frac{i}{2}}
		= \onehalf \left(\int_{\Omega_{F}} \Jn \abs{\bu^{n+\frac{i}{2}}}^{2} + \normg{\bv^{n+\frac{i}{2}}} + \normos{\bV^{n+\frac{i}{2}}} \right), \\
		& \ad{
			E_{el}^{n+1}
			& = \onehalf \left( \inner{\LE \bbeta^{n+\onehalf}}{\bbeta^{n+\onehalf}} + 2\mu_{S} \normos{\bD(\bd^{n+\onehalf})} + \normos{\nabla \cdot \bd^{n+\onehalf}} \right),
		} \\
		& E^{n+\frac{i}{2}} = E_{kin}^{n+\frac{i}{2}} + E_{el}^{n+1}, \\
		& \ad{
			D^{n+1}
			& = \kappa_{1} \Dt \int_{\Omega_{F}} \Jn \abs{\bD^{\tilde{\bbeta}^{n+1}}(\bu^{n+1})}^{p} + \frac{\Dt}{\alpha}\normg{(\bv^{n+1} - \bu^{n+1})_{\btau}}.
		}
	\end{align}
	
	Next, we give discrete energy estimates of each subproblem.
	\begin{lemma}\label{eststr}
		Solution of subproblem \eqref{weaks} satisfies the semi-discrete energy inequality
		\begin{equation}\label{eststr2}
			\ad{
				& E^{n+\onehalf}
				+ \onehalf \left( \normg{\bv^{n+\onehalf} - \bv^n} + \normos{\bV^{n+\onehalf} - \bV^n} \right) \\
				& \quad + \frac{\delta_0}{4} \normgh{\bbeta^{n+\onehalf} - \bbeta^n} + \onehalf \as{\bd^{n+\onehalf} - \bd^n}{\bd^{n+\onehalf} - \bd^n} \leq E^{n} + C,
			}
		\end{equation}
		where $ C $ is a constant independent of $ \Dt $.
	\end{lemma}
	\begin{proof}
		Let $ \bphi = \bv^{n+\onehalf} = \frac{\bbeta^{n+\onehalf} - \bbeta^{n}}{\Dt} $ and $ \bpsi = \bV^{n+\onehalf} = \frac{\bd^{n+\onehalf} - \bd^{n}}{\Dt} $ in \eqref{weaks}. More precisely, taking $ \bphi = \bv^{n+\onehalf} $ in the first term and $ \bpsi = \bV^{n+\onehalf} $ in the second term in \eqref{weaks}, $ \bphi = \frac{\bbeta^{n+\onehalf} - \bbeta^{n}}{\Dt} $ and $ \bpsi = \frac{\bd^{n+\onehalf} - \bd^{n}}{\Dt} $ in the other terms, by means of Remark \ref{identity}, we have
		\begin{align*}
			& \normg{\bv^{n+\onehalf}} + \normg{\bv^{n+\onehalf} - \bv^n} + \normos{\bV^{n+\onehalf}} + \normos{\bV^{n+\onehalf} - \bV^n} \\
			& + \innerl{\bbeta^{n+\onehalf}} + \frac{\delta_0}{2} \normgh{\bbeta^{n+\onehalf} - \bbeta^n} + \as{\bd^{n+\onehalf}}{\bd^{n+\onehalf}} \\
			& + \as{\bd^{n+\onehalf} - \bd^n}{\bd^{n+\onehalf} - \bd^n}
			\leq \normg{\bv^n} + \normos{\bV^n} + \innerl{\bbeta^n},
		\end{align*} 	
		where
		\begin{align*}
			& \int_{\Gamma} f(\bbeta^{n+\onehalf}) \cdot \bv^{n+\onehalf}
			= \int_{\Gamma} f(\bbeta^{n+\onehalf}) \cdot \frac{\bbeta^{n+\onehalf} - \bbeta^n}{\Dt} \\
			& \quad = \int_{\Gamma} \left( f(\bbeta^{n+\onehalf}) - f(\bbeta^{n}) \right) \cdot \frac{\bbeta^{n+\onehalf} - \bbeta^n}{\Dt}
			+ \int_{\Gamma} f(\bbeta^{n}) \cdot \frac{\bbeta^{n+\onehalf} - \bbeta^n}{\Dt} \\
			& \leq \frac{1}{\Dt} \norm{f(\bbeta^{n+\onehalf}) - f(\bbeta^{n})}_{L^2(\Gamma)} \norm{\bbeta^{n+\onehalf} - \bbeta^n}_{\bL^2(\Gamma)} + \frac{1}{\Dt} \norm{f(\bbeta^{n})}_{L^2(\Gamma)} \norm{\bbeta^{n+\onehalf} - \bbeta^n}_{\bL^2(\Gamma)} \\
			& \leq \frac{L^2 C_R}{\Dt} \norm{\bbeta^{n+\onehalf} - \bbeta^n}_{\bH^2(\Gamma)}^2
			+ \frac{2 L^4}{\delta_0 \Dt} \norm{f(\bbeta^{n})}_{H^2(\Gamma)}^2
			+ \frac{\delta_0}{8 \Dt} \norm{\bbeta^{n+\onehalf} - \bbeta^n}_{\bH^2(\Gamma)}^2 \\
			& \leq \frac{\delta_0}{4 \Dt} \norm{\bbeta^{n+\onehalf} - \bbeta^n}_{\bH^2(\Gamma)}^2 + \frac{C}{\Dt}.
		\end{align*}
		Here, we take $ C_R > 0 $ suitably small such that $ L^2 C_R \leq \delta_0 / 8 $ where $ L $ is the Poincar\'{e} constant in 1 dimensional space $ [0,L] $.
		By adding $ \int_{\Omega_{F}} J^n \abs{\bu^n}^2 $ on both sides, it follows from $ \bbeta^{n} = \bbeta^{n-\onehalf} $ and $ \bd^{n} = \bd^{n-\onehalf} $ in fluid subproblem that \eqref{eststr2} holds.
	\end{proof}
	\begin{lemma}\label{estflu}
		Solution of subproblem \eqref{fluidsubpro} satisfies the semi-discrete energy inequality
		\begin{equation}\label{estflu2}
			\ad{
				& E_{kin}^{n+1} + \onehalf \int_{\Omega_{F}}J^n \abs{\bu^{n+1} - \bu^{n+\onehalf}}^2 \\
				& \qquad + \onehalf \normg{\bv^{n+1} - \bv^{n+\onehalf}} + D_{N}^{n+1} \\
				& \quad \leq E_{kin}^{n+\onehalf} + C \Dt \left( \left( P_{\rm in}^n \right)^2 + \left( P_{\rm out}^n \right)^2 + 1 \right),
			}
		\end{equation}
		where $ C $ depends on $ \Omega_{F} $, $ \kappa_{1} $, $ \kappa_{2} $.
	\end{lemma}
	\begin{proof}
		Taking $ \bq = u^{n+1} $, $ \bphi = \bv^{n+1} $ in \eqref{weakf}, combining with \eqref{Coercivity} and using Remark \ref{identity}, we find that
		$$
			\ad{
				& \onehalf \int_{\Omega_{F}} J^n \left( \abs{\bu^{n+1}}^2 + \abs{\bu^{n+1} - \bu^{n+\onehalf}}^2 - \abs{\bu^{n+\onehalf}}^2 \right)
				+ \onehalf \int_{\Omega_{F}} \left( J^{n+1} - J^n \right) \abs{\bu^{n+1}}^2 \\
				& \quad + 2 \Dt \int_{\Omega_{F}} J^n \left( \kappa_{1} \abs{\bD^{\tetan{n+1}}(\bu^{n+1})}^p - \kappa_{2} \right)
				+ \frac{\Dt}{\alpha} \int_{\Gamma} \abs{\bu_{\tauF^{n+1}} - \bv_{\tauF^{n+1}}}^2 J_F^{n+1} \\
				& \quad + \onehalf \left( \normg{\bv^{n+1}} + \normg{\bv^{n+1} - \bv^n} - \normg{\bv^n} \right) \\
				& \leq \Dt  \left( P_{\rm in}^n (t) \int_{\Gamma_{\rm in}} \bu^{n+1} \cdot \nuF - P_{\rm out}^n (t) \int_{\Gamma_{\rm out}} \bu^{n+1} \cdot \nuF \right).
			}
		$$
		By using the trace inequality and Sobolev inequality, we obtain that for $ p > 2 $,
		\begin{align*}
			\abs{P_{\rm in}^n (t) \int_{\Gamma_{\rm in}} \bu^{n+1} \cdot \nuF}
			& \leq \abs{P_{\rm in}^n (t)} \abs{\int_{\Gamma_{\rm in}} \bu^{n+1} \cdot \nuF} \\
			& \leq c \abs{P_{\rm in}^n (t)} \norm{\bu^{n+1}}_{\bH^1(\Omega_{F}(t))} \\
			& \leq \frac{c}{2 \varepsilon} \abs{P_{\rm in}^n (t)}^2 + \frac{c \varepsilon}{2} \norm{ \bu^{n+1} }_{\bH^1(\Omega_{F}(t))}^2 \\
			& \leq \frac{c}{4 \varepsilon} \abs{P_{\rm in}^n (t)}^2 + \frac{c^* \varepsilon}{2} \norm{ \bu^{n+1} }_{\bW^{1,p}(\Omega_{F}(t))}^p.
		\end{align*}
		Moving $ 2 \Dt \int_{\Omega_{F}} J^n \kappa_{2} $ to the right-hand side of above inequality and proceeding as in \cite{MC2014JDE,MC2016JDE}, we choose $ \varepsilon > 0 $ small enough such that $ c^* \varepsilon \leq \kappa_{1} $ to complete the proof.
	\end{proof}
	
	Subsequently, we obtain the uniform energy estimates in the following. We \textbf{add the subscript $ N $ to all variables} so that we can pass to the limit as $ N \rightarrow \infty $ in next section.
	
	\begin{lemma}[Uniform energy estimates]\label{estunif}
		Let $ \Delta t > 0 $, $ N = T/\Delta t $ and $ 0 < T \leq T^* $, then we have the following estimates:
		\begin{enumerate}
			\item[$ \mathrm{1.} $] $ \EN{n+\onehalf} \leq K $, $ \EN{n+1} \leq K $, for all $ n = 0, 1, \cdots, N - 1. $
			\item[$ \mathrm{2.} $] $ \displaystyle \sum_{j=1}^{N}\DN{j} \leq K $.
			\item[$ \mathrm{3.} $] $ \displaystyle \sum_{n=0}^{N-1}\left( \int_{\Omega_{F}} \JN{n} \abs{\bu_N^{n+1} - \bu_N^{n}}^{2} + \normg{\bv_N^{n+1} - \bv_N^{n+\onehalf}} \right. \\
			\left. \qquad\qquad + \normg{\bv_N^{n+\onehalf} - \bv_N^{n}} + \normos{\bV_N^{n+1} - \bV_N^{n}} \right) \leq K, $
			\item[$ \mathrm{4.} $] $ \displaystyle \sum_{n=0}^{N-1}\left( \normgh{\bbeta_N^{n+1} - \bbeta_N^{n}} + \as{\bd_N^{n+1} - \bd_N^{n}}{\bd_N^{n+1} - \bd_N^{n}} \right) \leq K, $
		\end{enumerate}
	where $ K = E_{0} + C \left( \norm{P_{\rm in}}_{L^2(0,T)}^2 + \norm{P_{\rm out}}_{L^2(0,T)}^2 + T^* \right) $ and $ C $ depends on $ \Omega_{F} $ and other parameters in the problem.
	\end{lemma}
	\begin{proof}
		Adding \eqref{eststr2} and \eqref{estflu2}, and taking sum for $ n $ from $ 0 $ to $ N-1 $, we find that
		\begin{equation}\label{sumofE}
			\ad{
				& \EN{N} + \sum_{n=0}^{N-1} \DN{n+1}
				+ \onehalf \sum_{n=0}^{N-1} \biggl( \int_{\Omega_{F}} \JN{n} \abs{\un{n+1} - \un{n}}^2
				+ \normg{\vn{n+1} - \vn{n+\onehalf}}  \biggr.\\
				& \quad \biggl.
				+ \normg{\vn{n+\onehalf} - \vn{n}}
				+ \normos{\VN{n+1} - \VN{n}} \biggr.\\
				& \quad \biggl.
				+ \normgh{\etan{n+1} - \etan{n}}
				+ \as{\dn{n+1} - \dn{n}}{\dn{n+1} - \dn{n}} \biggr) \biggr.\\
				& \biggl.
				\leq E_{0} + C \sum_{n=0}^{N-1} \Dt \left( \left( P_{\rm in}^n \right)^2 + \left( P_{\rm out}^n \right)^2 + 1 \right).
			}
		\end{equation}
		By the definition of $ P_{\rm in/out}^n $ and the H\"{o}lder's inequality, we obtain
		\begin{align*}
			\sum_{n=0}^{N-1} \Dt \left( P_{\rm in}^n \right)^2
			& = \sum_{n=0}^{N-1} \Dt \left( \frac{1}{\Dt} \int_{n\Dt}^{(n+1)\Dt} P_{\rm in} \right)^2 = \frac{1}{\Dt} \sum_{n=0}^{N-1} \left( \int_{n\Dt}^{(n+1)\Dt} P_{\rm in} \right)^2 \\
			& \leq \frac{1}{\Dt} \sum_{n=0}^{N-1} \int_{n\Dt}^{(n+1)\Dt} P_{\rm in}^2 \int_{n\Dt}^{(n+1)\Dt} 1^2 = \norm{P_{\rm in}}_{L^2(0,T)}^2.
		\end{align*}
		Then the last term in \eqref{sumofE} becomes
		\begin{align*}
			C \sum_{n=0}^{N-1} \Dt \left( \left( P_{\rm in}^n \right)^2 + \left( P_{\rm out}^n \right)^2 + 1 \right) \leq C \left( \norm{P_{\rm in}}_{L^2(0,T)}^2 + \norm{P_{\rm out}}_{L^2(0,T)}^2 + T^* \right).
		\end{align*}
		Therefore, we have the estimates 2, 3, and 4.
		It can be deduced form summing from $ 0 $ to $ n $ instead of from $ 0 $ to $ N-1 $ that $ \EN{n+1} \leq K $. Next, \eqref{eststr2} implies $ \EN{n+\onehalf} \leq \EN{n} \leq K $.
	\end{proof}
	
%
	
\subsection{Uniform boundedness}
	\label{uniformboundedness}
	
	Let $ \etan{n} $, $ \vn{n-\onehalf} $, $ \dn{n} $, $ \VN{n} $ be the solutions of \eqref{weaks} given in Lemma \ref{structurepro} and $ \un{n} $, $ \vn{n} $ be the solutions of \eqref{fluidsubpro} given in Lemma \ref{fluidpro}. Now we are in position to give a sequence of approximate solution on $ (0,T) $. Following the procedure in \cite{CGM2020,MC2016JDE}, we define approximate solutions on each time subinterval $ (n\Dt, (n+1)\Dt] $ as the \textbf{piecewise constant functions}:
	\begin{align*}
		& \uNn(t, \cdot) = \un{n+1},
		\ \vNn(t, \cdot) = \vn{n+1},
		\ \etaNn(t, \cdot) = \etan{n+1}, \\
		& \dNn(t, \cdot) = \dn{n+1},
		\ \VNn(t, \cdot) = \VN{n+1},
		\ \tvNn(t, \cdot) = \vn{n+\onehalf},	
	\end{align*}
	for $ n = 0, 1, \cdots, N - 1 $.
	Here we use $ \tvNn(t, \cdot) = \vn{n+\onehalf} $ to represent the thin elastic structure velocity in structure subproblem and $ \vNn(t, \cdot) = \vn{n+1} $ to represent the thin elastic structure velocity in fluid subproblem.	Since in fluid subproblem, the trace of $ \un{n+1} $ in $ \Gamma(t) $ is $ \vn{n+1} $ and it is different from the ``unchanged'' displacement velocity
	\begin{align*}
		\frac{\etan{n+1} - \etan{n}}{\Dt} = \frac{\etan{n+\onehalf} - \etan{n}}{\Dt} = \vn{n+\onehalf}.
	\end{align*}
	In the subsequent Lemma \ref{weakstarcon}, we will show that they are exactly the same in some sense when we pass to the limit as $ N \rightarrow \infty $.
	
	Analogously, for $ n = 0, 1, \cdots, N - 1 $, other approximate quantities can be defined as
	\begin{align*}
		& \ANn(t, \cdot) = \AN{n+1},
		\ \nuFN(t, \cdot) = \nuF^{n+1},
		\ \tauFN(t, \cdot) = \tauF^{n+1}, \\
		& \wNn(t, \cdot) = \wn{n+1},
		\ J_{F,N}(t, \cdot) = J_F^{n+1},
		\ \JNN(t, \cdot) = \JN{n+1}.
	\end{align*}
	
	As discussed in Section \ref{ALE}, we need to determine the time interval of existence of solutions. With compatibility conditions \eqref{u0}--\eqref{eta0}, we have the following proposition due to the construction of $ \AN{n+1} $ in Section \ref{fluidsubproblem}:
	\begin{proposition}[\cite{MC2016JDE}]
		\label{existencetime}
		There is a $ T^* > 0 $ small enough and a positive constant $ C $ such that $ \AN{n+1} $ is an injection, and
		$$
		\JN{n+1} = \det \nabla \AN{n+1} > 0, \quad \Dt = \frac{T^*}{N} > 0 ,\  n = 0, 1, \cdots, N - 1.
		$$
	\end{proposition}

	From Proposition \ref{existencetime}, we have the maximal existence time $ T^* $  before which the domain will not degenerate. Then Lemma \ref{estunif} implies the following boundedness properties.
	
	\begin{lemma}\label{bound}
		For a fixed $ \Dt = \frac{T}{N} > 0 $, $ 0 < T \leq T^* $ and $ p > 2 $, the following holds:
		\begin{enumerate}\label{boundsolutions}
			\item[$ \mathrm{1.} $] The sequence $ \setaNn $ is uniformly bounded in $ L^{\infty}(0,T;\bH^{2}_{0}(\Gamma)) $,
			\item[$ \mathrm{2.} $] The sequence $ \stvNn $ is uniformly bounded in $ L^{\infty}(0,T; \bL^{2}(\Gamma)) $,
			\item[$ \mathrm{3.} $] The sequence $ \sdNn $ is uniformly bounded in $ L^{\infty}(0,T; \bH^{1}(\Omega_{S})) $,
			\item[$ \mathrm{4.} $] The sequence $ \sVNn $ is uniformly bounded in $ L^{\infty}(0,T; \bL^{2}(\Omega_{S})) $,
			\item[$ \mathrm{5.} $] The sequence $ \suNn $ is uniformly bounded in $ L^{\infty}(0,T; \bL^{2}(\Omega_{F})) $,
			\item[$ \mathrm{6.} $] The sequence $ \svNn $ is uniformly bounded in $ L^{\infty}(0,T; \bL^{2}(\Gamma)) $.
		\end{enumerate}
	\end{lemma}
	
	Let us denote $ \tetaNn = \TN \etaNn(t, \cdot) = \etaNn(t - \Dt, \cdot) $, $ \Dt = T/N $. We can deduce the following lemma.
	\begin{lemma}\label{boundDU}
		For a pair of conjugate indices $ p $ and $ q $ which satisfies $ \frac{1}{p} + \frac{1}{q} = 1 $, we have
		\begin{enumerate}
			\item[$ \mathrm{1.} $] The sequence $ \left\{ \bD^{\tetaNn}(\uNn) \right\}_{N \in \mathbb{N}} $ is uniformly bounded in $ \bL^{p}((0,T) \times \Omega_{F})^2, $
			\item[$ \mathrm{2.} $] The sequence $ \left\{ \bbS(\bD^{\tetaNn}(\uNn)) \right\}_{N \in \mathbb{N}} $ is uniformly bounded in $ \bL^{q}((0,T) \times \Omega_{F})^{2}. $
		\end{enumerate}
	\end{lemma}
	\begin{proof}
		The estimate 2 in Lemma \ref{estunif} implies that
		$$
			\sum_{n = 0}^{N - 1} \kappa_{1} \int_{\Omega_{F}} \Jn_{N} \abs{\bD^{\tetan{n+1}} (\un{n+1})}^{p} \Dt \leq K.
		$$
		Since the Jacobian determinant $ \JN{n} $ of the ALE mapping $ \ANn^n $ is uniformly bounded from below by a positive constant (see Proposition \ref{existencetime}), we have the boundedness of $ \bD^{\tetaNn}(\uNn) $. Combining this boundedness with \eqref{Boundedness}, we obtain
		\begin{equation*}
			\norm{\bbS \left( \bD^{\tetaNn}(\uNn) \right)}_{ \bL^{q}((0,T) \times \Omega_{F})^{2} }
			\leq K'.
		\end{equation*}
		This completes the proof.
	\end{proof}	
\subsection{Weak and weak* convergence}
	\label{weakconvergence}
	\begin{lemma}[Weak* convergence]\label{weakstarcon}
		For a fixed $ \Dt = \frac{T}{N} > 0 $, $ 0 < T \leq T^* $, there exist subsequences $ \setaNn $, $ \stvNn $, $ \sdNn $, $ \sVNn $, $ \suNn $ and $ \svNn $ and functions $ \bbeta \in L^{\infty}(0,T; \bH^{2}_{0}(\Gamma)) $, $ \bv, \btv \in L^{\infty}(0,T; \bL^{2}(\Gamma)) $, $ \bu \in L^{\infty}(0,T; \bL^{2}(\Omega_{F})) $, $ \bd \in L^{\infty}(0,T; \bH^{1}(\Omega_{S})) $ and $ \bV \in L^{\infty}(0,T; \bL^{2}(\Omega_{S})) $ such that
		\begin{align*}
			\etaNn \rightharpoonup \bbeta &\text{ weakly* in } L^{\infty}(0,T;\bH^{2}_{0}(\Gamma)), \\
			\tvNn \rightharpoonup \btv &\text{ weakly* in } L^{\infty}(0,T;\bL^{2}(\Gamma)), \\
			\dNn \rightharpoonup \bd &\text{ weakly* in } L^{\infty}(0,T; \bH^{1}(\Omega_{S})), \\
			\VNn \rightharpoonup \bV &\text{ weakly* in } L^{\infty}(0,T; \bL^{2}(\Omega_{S})),\\
			\uNn \rightharpoonup \bu &\text{ weakly* in } L^{\infty}(0,T;\bL^{2}(\Omega_{F})),\\
			\vNn \rightharpoonup \bv &\text{ weakly* in } L^{\infty}(0,T;\bL^{2}(\Gamma)).
		\end{align*}
		Furthermore,
		\begin{align*}
			\bv = \btv.
		\end{align*}
	\end{lemma}
	\begin{proof}
		From the uniform boundedness in Lemma \ref{bound}, we get the weak* convergences. It only need to prove $ \bv = \btv $. By the definition of approximate solutions and estimate 3 in Lemma \ref{estunif}, we have
		\begin{align*}
			\norm{\vNn - \tvNn}_{L^2(0,T; \bL^2(\Gamma))}^2
			& = \int_0^T \norm{\vNn - \tvNn}_{\bL^2(\Gamma)}^2
			= \sum_{n=0}^{N-1} \int_{n\Dt}^{(n+1)\Dt} \normg{\bv_N^{n+1} - \bv_N^{n+\onehalf}}  \\
			& = \sum_{n=0}^{N-1} \normg{\bv_N^{n+1} - \bv_N^{n+\onehalf}} \Dt \leq K \Dt.
		\end{align*}
		For a function $ \bphi \in L^2(0,T; \bL^2(\Gamma)) $, setting $ \Dt \rightarrow 0 $ and combining the weak* convergences of $ \vNn $ and $ \tvNn $, we obtain
		\begin{align*}
			\intt \int_{\Gamma} \left(\bv - \btv\right) \cdot \bphi
			& = \intt \int_{\Gamma} \left(\bv - \vNn\right) \cdot \bphi
			+ \intt \int_{\Gamma} \left(\vNn - \tvNn\right) \cdot \bphi
			+ \intt \int_{\Gamma} \left(\tvNn - \btv\right) \cdot \bphi \\
			& \leq \intt \int_{\Gamma} \left(\bv - \vNn\right) \cdot \bphi
			+ \norm{\vNn - \tvNn}_{L^2(0,T; \bL^2(\Gamma))} \norm{\bphi}_{L^2(0,T; \bL^2(\Gamma))} \\
			& \quad	+ \intt \int_{\Gamma} \left(\tvNn - \btv\right) \cdot \bphi \\
			& \rightarrow 0,
		\end{align*}
		which means that $ \bv = \btv $ a.e. in $ (0,T) \times \Gamma $.
	\end{proof}

	By means of the reflexivity and Lemma \ref{boundDU}, we have the weak convergence results.
	\begin{lemma}[Weak convergence]\label{weakcon}
		For a fixed $ \Dt = \frac{T}{N} > 0 $, $ 0 < T \leq T^* $, $ p > 2 $ and $ q $ which satisfies $ \frac{1}{p} + \frac{1}{q} = 1 $, there exist subsequences $ \left\{ \bD^{\tetaNn}(\uNn) \right\}_{N \in \mathbb{N}} $ and $ \left\{ \bbS(\bD^{\tetaNn}(\uNn)) \right\}_{N \in \mathbb{N}} $ and functions $ \bm{M} \in \bL^p((0,T) \times \Omega_{F})^2 $, $ \bm{G} \in \bL^q((0,T) \times \Omega_{F})^2 $ such that
		\begin{equation*}
			\ad{
				\bD^{\tetaNn}(\uNn) \rightharpoonup \bm{M} &\text{ weakly in } \bL^p((0,T) \times \Omega_{F})^2, \\
				\bbS(\bD^{\tetaNn}(\uNn)) \rightharpoonup \bm{G} &\text{ weakly in } \bL^q((0,T) \times \Omega_{F})^2.
			}
		\end{equation*}
	\end{lemma}
	\begin{remark}
		Here, $ \bm{M} $ and $ \bm{G} $ are unknown since the gradient are not equal in $ \nabla^{\tetaNn} \uNn $ and $ \nabla \bu $. In the limiting process, we can show that $ \bm{M} = \bD^{\bbeta} (\bu) $ and $ \bm{G} = \bbS(\bD^{\bbeta} (\bu)) $ (see Lemma \ref{gradientscon}).
	\end{remark}
	
\subsection{Strong convergence}
	\label{strongconvergence}
	In this section, we prove the strong convergence of weak solution. This convergence is useful when we finally pass to the limit.
\subsubsection{Strong convergence for velocities}
	First, we establish the strong convergence of $ \uNn $ and $ \vNn $. In \cite{MC2019JDE}, Muha and \v{C}ani\'{c} proved a generalized Aubin-Lions-Simon Lemma to deal with the specific problems for which the spatial domain depends on time. More precisely, they made use of the classical Simon's theorem in \cite[Theorem 1]{Simon1987} together with the uniform estimates of the problem and provided the $ L^{2}(0,T;H) $ compactness for the moving domains. This theorem is effective in processing this type of problem and Muha and \v{C}ani\'{c} gave three examples in \cite{MC2019JDE}, whose strong convergences were proved by a more sophisticated procedure before. After \cite{MC2019JDE}, the generalized Aubin-Lions-Simon Lemma was applied in several works, see \cite{CGM2020,MS2019,TW2018,TW2020}.
	
	To do this, we give the generalized Aubin-Lions-Simon compactness lemma for problems on moving domains \cite{MC2019JDE}.
	\begin{theorem}[{\cite[Theorem 3.1]{MC2019JDE}}]\label{compact}
		Let $ V $, $ H $ be Hilbert spaces such that $ V \subset \subset H $. Suppose that $ \{ \uNn \} \subset L^2 (0,T; H) $ is a sequence such that $ \uNn(t, \cdot) = \uNn^n (\cdot) $ on $ ( (n - 1)\Dt, n \Dt ] $, $ n = 1, \dots, N $, with $ N \Dt = T $. Let $ V_{N}^n $ and $ Q_{N}^n $ be Hilbert spaces such that $ ( V_{N}^n, Q_{N}^n ) \hookrightarrow V \times V $, where the embeddings are uniformly continuous w.r.t. $ \Dt $ and $ n $, and $ V_{N}^n \subset \subset \overline{Q_{N}^n}^{H} \hookrightarrow (Q_{N}^n)' $. Let $ \uNn^n \in V_{N}^n $, $ n = 1, \dots, N $. If the following is true:
		\begin{enumerate}[label=(\Alph*)]
			\item \label{A} There exists a universal constant $ C > 0 $ such that for every $ \Dt $
			\begin{enumerate}[label=(A\arabic*)]
				\item \label{A1} $ \sum_{n = 1}^{N} \norm{\uNn^n}_{V_{N}^n}^2 \Dt \leq C $,
				\item \label{A2} $ \norm{\uNn}_{L^{\infty}(0,T; H)} \leq C $.
			\end{enumerate}
			\item \label{B}There exists a universal constant $ C > 0 $ such that
			$$
				\norm{P_{N}^n \frac{\un{n + 1} - \un{n}}{\Dt}}_{( Q_{N}^n )'}
				\leq C \left( \norm{\un{n + 1}}_{V_{N}^{n+1}} + 1 \right), n = 0, \dots, N - 1,
			$$
			where $ P_{N}^n $ is the orthogonal projector onto $ \overline{Q_{N}^n}^{H} $.
			\item \label{C}The function spaces $ Q_{N}^n $ and $ V_{N}^n $ depend smoothly on time in the following sense:
			\begin{enumerate}[label=(C\arabic*)]
				\item \label{C1} For every $ \Dt > 0 $, and for every $ l \in \{ 1, \dots, N - 1 \} $, there exists a space $ Q_{N}^{n,l} \subset V $ and the operators $ J_{N,l,n}^i : Q_{N}^{n.l} \rightarrow Q_{N}^{n+i}, i = 0,1, \dots, l $, such that $ \norm{J_{N,l,n}^i \bq}_{Q_{N}^{n + i}} \leq C \norm{\bq}_{Q_{N}^{n,l}} $, $ \forall \  \bq \in \Q_{N}^{n,l} $, and
				\begin{align}
					& \left( \left( J_{N,l,n}^{j + 1} \bq - J_{N,l,n}^{j} \bq \right), \un{n + j + 1} \right)_{H} \\
					& \qquad\qquad \leq C \Dt \norm{\bq}_{Q_{N}^{n,l}} \norm{\un{n + j + 1}}_{V_{N}^{n + j + 1}}, \quad j \in \{ 0, \dots, l - 1 \},\\
					& \norm{J_{N,l,n}^i \bq - \bq}_{H} \leq C \sqrt{l \Dt} \norm{\bq}_{Q_{N}^{n,l}}, \quad i \in \{ 0, \dots, l \},
				\end{align}
				where $ C > 0 $ is independent of $ \Dt $, $ n $ and $ l $.
				\item \label{C2} Let $ V_{N}^{n,l} = \overline{Q_{N}^{n,l}}^{V} $. There exist the functions $ I_{N,l,n}^{i} : V_{N}^{n + i} \rightarrow V_{N}^{n,l} $, $ i = 0, 1, \dots, l $, and a universal constant $ C > 0 $, such that for every $ \bv \in V_{N}^{n + i} $
				\begin{gather}
					\norm{I_{N,l,n}^{i} \bv}_{V_{N}^{n,l}} \leq C \norm{\bv}_{V_{N}^{n+i}}, \quad i \in \{ 0, \dots, l \}, \\
					\norm{I_{N,l,n}^{i} \bv - \bv}_{H} \leq g(l\Dt) \norm{\bv}_{V_{N}^{n + i}}, \quad i \in \{ 0, \dots, l \}.
				\end{gather}
				where $ g :\mathbb{R}_{+} \rightarrow \mathbb{R}_{+} $ is a universal, monotonically increasing function such that $ g(h) \rightarrow 0 $ as $ h \rightarrow 0 $.
				\item \label{C3} Uniform Ehrling property: For every $ \delta > 0 $, there exists a constant $ C(\delta) $ independent of $ n $, $ l $ and $ \Dt $, such that
				\begin{equation}
					\norm{\bv}_{H} \leq \delta \norm{\bv}_{V_{N}^{n,l}} + C(\delta) \norm{\bv}_{(Q_{N}^{n,l})'}, \quad \bv \text{ in } Q_{N}^{n.l};
				\end{equation}
			\end{enumerate}
		\end{enumerate}
		Then $ \{ \uNn \} $ is relatively compact in $ L^{2}(0,T; H) $.
	\end{theorem}
	
	With the help of Theorem \ref{compact}, we have the following compactness theorem:
	\begin{theorem}\label{compactnessuv}
		The sequence $ \left\{ \left( \uNn, \vNn, \VNn \right) \right\}_{N \in \mathbb{N}} $, introduced in Lemma \ref{weakstarcon}, is relatively compact in $ L^{2}(0,T; H) $, where $ H = \bL^{2}(\OM) \times \bH^{-s}(\Gamma) \times \bH^{-s}(\Omega_{S}) $, $ \OM $ is the union of all parameterized fluid domains and $ 0 < s < \onehalf $.
	\end{theorem}
	\begin{proof}
		Since the model of fluid-thin structure interaction in our work can be found in \cite{MC2016JDE} and \cite[Section 4.3]{MC2019JDE}, the proof of Theorem \ref{compactnessuv} can be easily established by Theorem \ref{compact} with some modifications. We need to take $ L^{p}(0,T; \bW^{1,p}) \hookrightarrow L^{2}(0,T; \bH^{1}) $ into account and make use of the estimates in Lemma \ref{estunif} to verify all the conditions of Theorem \ref{compact}, more specifically, Properties \ref{A}, \ref{B} and \ref{C}. Here, we give a sketch of the proof.
		
		\textbf{Property \ref{A}.} To show \ref{A1} and \ref{A2} of Properties \ref{A}, we define the corresponding spaces as follows.
		\begin{align*}
			V = \bH^s(\OM) \times \bL^2(\Gamma) \times \bL^2(\Omega_{S}),
		\end{align*}
		and
		\begin{align*}
			H = \bL^{2}(\OM) \times \bH^{-s}(\Gamma) \times \bH^{-s}(\Omega_{S}), \quad 0 < s < \onehalf.
		\end{align*}
		Then we have $ V \subset \subset H $. In addition, we choose the moving velocity spaces $ V_N^n $ and moving test spaces $ Q_N^n $ as
		\begin{align*}
			V_N^n = \left\{ (\bu, \bv, \bV) \in \overline{V_{F}^{n}}^{\bH^{1}(\Omega_{F}^n)} \times \bH^{\onehalf}(\Gamma) \times \bL^2(\Omega_{S}): \left( \rv{\bu}_{\Gamma^{n-1}} - \bv \right) \cdot \nuFN^{n-1} = 0 \right\},
		\end{align*}
		\begin{align*}
			Q_N^n = \left\{ (\bq, \bphi, \bpsi) \in \left( \overline{V_{F}^{n}}^{\bH^{1}(\Omega_{F}^n)} \cap \bH^4(\Omega_{F}^n) \right) \times \V_{W} \times \V_{S}: \rv{\bq}_{\Gamma^n} = \bphi, \bphi = \rv{\bpsi}_{\Gamma} \right\}.
		\end{align*}
		such that $ (V_N^n, Q_N^n) \hookrightarrow V \times V $. It follows from the trace theorem $ \norm{\vn{n}}_{\bH^{\onehalf}(\Gamma)}^2 \leq \norm{\un{n}}_{\bH^{1}(\Omega_{F}^n)}^2 $ and Lemma \ref{estunif} that
		\begin{align}
			& \quad \  \sum_{n=1}^{N} \norm{(\un{n}, \vn{n}, \VN{n})}_{V_N^n}^2 \Dt \nonumber\\
			& = \sum_{n=1}^{N} \left( \norm{\un{n}}_{\bH^{1}(\Omega_{F}^n)}^2 + \norm{\vn{n}}_{\bH^{\onehalf}(\Gamma)}^2 + \norm{\VN{n}}_{\bL^{2}(\Omega_{S})}^2 \right) \Dt \nonumber\\
			& \leq C \sum_{n=1}^{N} \left( \norm{\un{n}}_{\bH^{1}(\Omega_{F}^n)}^2 + \norm{\VN{n}}_{\bL^{2}(\Omega_{S})}^2 \right) \Dt \label{uvvdt} \\
			& \leq C.\nonumber
		\end{align}
		From the embedding $ \bL^2 \hookrightarrow \bH^{-s} $, $ 0 < s < \onehalf $ and the uniform boundedness in Lemma \ref{bound}, we deduce that
		\begin{align}
			& \quad \ \norm{(\uNn, \vNn, \VNn)}_{L^{\infty}(0,T; H)}^2 \nonumber\\
			& = \norm{\uNn}_{L^{\infty}(0,T; \bL^2(\OM))}^2
			+ \norm{\vNn}_{L^{\infty}(0,T; \bH^{-s}(\Gamma))}^2
			+ \norm{\VNn}_{L^{\infty}(0,T; \bH^{-s}(\Omega_{S}))}^2 \label{uvvh}\\
			& \leq \norm{\uNn}_{L^{\infty}(0,T; \bL^2(\OM))}^2
			+ \norm{\vNn}_{L^{\infty}(0,T; \bL^2(\Gamma))}^2
			+ \norm{\VNn}_{L^{\infty}(0,T; \bL^2(\Omega_{S}))}^2 \nonumber\\
			& \leq C.\nonumber
		\end{align}
		Then \ref{A1} and \ref{A2} are verified by \eqref{uvvdt} and \eqref{uvvh} respectively.
		
		\textbf{Property \ref{B}.} In the following, we prove the Property \ref{B}. First, we give the weak formulation in the moving domain $ \Omega_{F}^n $:
		\begin{align}
			& \quad \onehalf \int_{\Omega_{F}^n} \bigg(
			\left( \left( \hun{n} - \wn{n+1} \right) \cdot \nabla \right) \un{n+1} \cdot \bq
			- \left( \left( \hun{n} - \wn{n+1} \right) \cdot \nabla \right) \bq \cdot \un{n+1}
			\bigg) \nonumber \\
			& + \int_{\Omega_{F}^n} \frac{\un{n+1} - \hun{n}}{\Dt} \cdot \bq
			+  \onehalf \int_{\Omega_{F}^n} \frac{\JN{n+1} - \JN{n}}{\JN{n} \Dt} \un{n+1} \cdot \bq \nonumber \\
			& + 2 \int_{\Omega_{F}^n} \bbS(\bD(\un{n+1})) : \bD(\bq)
			- \int_{\Gamma} \frac{\vn{n+1} - \vn{n}}{\Dt} \cdot \bphi
			+ \inner{f(\etan{n+1})}{\bphi} \label{weakformmove} \\
			& + \frac{1}{\alpha} \int_{\Gamma} \left( \un{n+1} \cdot \tauF^{n+1} - \vn{n+1} \cdot \tauF^{n+1} \right) \bq \cdot \tauF^{n+1} J_{F}^{n+1} \dz \dt
			+ \inner{\LE \etan{n+1}}{\bphi}  \nonumber \\
			& + \frac{1}{\alpha} \int_{\Gamma} \left( \vn{n+1} \cdot \tauF^{n+1} - \un{n+1} \cdot \tauF^{n+1} \right) \bphi \cdot \tauF^{n+1} J_{F}^{n+1} \dz \dt  \nonumber \\
			& - \int_{\Omega_{S}} \frac{\VN{n+1} - \VN{n}}{\Dt} \cdot \bpsi
			+ \as{\dn{n+1}}{\bpsi} = \int_{\Gamma_{\rm in/out}} P_{\rm in/out}^{n+1} \bq \cdot \nuF, \nonumber
		\end{align}
		for all $ (\bq,\bphi,\bpsi) \in \Q_{N}^n $, where $ \hun{n} = \un{n} \circ \left( \ANn^{n-1} \right)^{-1} $ and $ \frac{1}{\JN{n}} $ is the Jacobian of the mapping $ \left( \AN{n} \right)^{-1} $.
		
		A direct calculation yields
		\begin{align*}
			& \quad \ \norm{P_N^{n+1} \frac{(\un{n+1}, \vn{n+1}, \VN{n+1}) - (\un{n}, \vn{n}, \VN{n})}{\Dt} }_{(Q_N^{n+1})'} \\
			& = \sup_{\norm{(\bq, \bphi, \bpsi)}_{Q_N^{n+1}} = 1}
			\abs{
				\int_{\Omega_{F}^{n}} \frac{\un{n+1} - \un{n}}{\Dt} \cdot \bq
				+ \int_{\Gamma} \frac{\vn{n+1} - \vn{n}}{\Dt} \cdot \bphi
				+ \int_{\Omega_{S}} \frac{\VN{n+1} - \VN{n}}{\Dt} \cdot \bpsi
			} \\
			& \leq \sup_{\norm{(\bq, \bphi, \bpsi)}_{Q_N^{n+1}} = 1}
			\abs{
				\int_{\Omega_{F}^{n}} \frac{\un{n+1} - \hun{n}}{\Dt} \cdot \bq
				+ \int_{\Gamma} \frac{\vn{n+1} - \vn{n}}{\Dt} \cdot \bphi
				+ \int_{\Omega_{S}} \frac{\VN{n+1} - \VN{n}}{\Dt} \cdot \bpsi
			} \\
			& \qquad \qquad + \sup_{\norm{(\bq, \bphi, \bpsi)}_{Q_N^{n+1}} = 1}
			\abs{
				\int_{\Omega_{F}^{n}} \frac{\hun{n} - \un{n}}{\Dt} \cdot \bq
			}.
		\end{align*}
		By means of the weak formulation \eqref{weakformmove}, we obatin
		\begin{align*}
			& \quad \ \abs{
				\int_{\Omega_{F}^{n}} \frac{\un{n+1} - \hun{n}}{\Dt} \cdot \bq
				+ \int_{\Gamma} \frac{\vn{n+1} - \vn{n}}{\Dt} \cdot \bphi
				+ \int_{\Omega_{S}} \frac{\VN{n+1} - \VN{n}}{\Dt} \cdot \bpsi
			} \\
			& \leq \left( \norm{\hun{n}} + \norm{\bw^{n+1}} \right) \norm{\nabla \un{n+1}} \norm{\bq}_{\bL^{\infty}}
			+ \left( \norm{\hun{n}} + \norm{\bw^{n+1}} \right) \norm{\nabla \bq} \norm{\un{n+1}} \\
			& \quad + C \norm{\tilde{\bv}_N^{n+1}} \norm{\un{n+1}} \norm{\bq}_{L^\infty} + \left( \norm{\bD^{\tetan{n+1}}(\un{n+1})}_{\bL^p}^p + \norm{\bD^{\tetan{n+1}}(\un{n+1})} \right) \norm{\bD(q)} \\
			& \quad + \norm{\Delta \bbeta} \norm{\Delta \phi}
			+ C \norm{\bbeta}_{\bH^{2-\varepsilon}} \norm{\phi}_{\bH^2}
			+ \norm{f(\bm{0})}_{H^{2}} \norm{\bphi}_{\bH^2}\\
			& \quad + \norm{\un{n+1} \cdot \tauF^{n+1} - \vn{n+1} \cdot \tauF^{n+1}} \norm{\bq \cdot \tauF^{n+1} - \bphi \cdot \tauF^{n+1} }
			+ \norm{\bd}_{\bH^1} \norm{\bpsi}_{\bH^1} \\
			& \leq C \left( \norm{(\un{n+1}, \vn{n+1}, \VN{n+1})}_{V_N^{n+1}} + 1 \right) \norm{(\bq, \bphi, \bpsi)}_{(Q_N^{n+1})},
		\end{align*}
		where we used assumption \ref{f1} that
		\begin{align*}
			\norm{f(\bbeta)}_{\bH^{2}}
			& = \norm{f(\bbeta) - f(\bm{0}) + f(\bm{0})}_{\bH^{2}} \\
			& \leq \norm{f(\bbeta) - f(\bm{0})}_{\bH^{2}}
			+ \norm{f(\bm{0})}_{\bH^{2}} \\
			& \leq C \norm{\bbeta}_{\bH^{2-\varepsilon}} + \norm{f(\bm{0})}_{H^{2}}.
		\end{align*}
		Following the same procedure in \cite{CGM2020} and \cite{MC2019JDE}, we get
		\begin{align*}
			\abs{
				\int_{\Omega_{F}^{n}} \frac{\hun{n} - \un{n}}{\Dt} \cdot \bq
			} \leq C \norm{(\un{n}, \vn{n}, \VN{n})}_{V_N^{n}} \norm{(\bq, \bphi, \bpsi)}_{(Q_N^{n})}.
		\end{align*}
		Consequently,
		\begin{align*}
			& \quad \ \norm{P_N^{n+1} \frac{(\un{n+1}, \vn{n+1}, \VN{n+1}) - (\un{n}, \vn{n}, \VN{n})}{\Dt} }_{(Q_N^{n+1})'} \\
			& \leq C \left( \norm{(\un{n}, \vn{n}, \VN{n})}_{V_N^{n}} + 1 \right),
		\end{align*}
		which proves Property \ref{B}.
		
		\textbf{Property \ref{C}.} We notice that the rest properties needed to be shown follow from the same procedure in \cite{CGM2020}. Therefore, we provide the definition of several operators and spaces.
		First, we denote the operator $ J_{N,l,n}^i: Q_N^{n,l} \rightarrow Q_N^{n+i} $ by
		\begin{align*}
			J_{N,l,n}^i(\bq, \bphi, \bpsi) = \left( \rv{\bq}_{\Omega_{F}^{n+i}}, \rv{\bq}_{\Gamma^{n+i}}, \bpsi \right),
		\end{align*}
		and space $ Q_N^{n,l} $ by
		\begin{align*}
			Q_N^{n,l} = \left\{ (\bq, \bphi, \bpsi) \in \left( V_F^{n,l} \cap \bH^4(\Omega_{F}^{n,l}) \right) \times \V_{W} \times \V_{S}: \rv{\bq}_{\Gamma^{n}} \cdot \nuF^n = \bphi \cdot \nuF^n, \bphi = \rv{\bpsi}_{\Gamma} \right\}.
		\end{align*}
		Moreover, to establish Property (C2), we need the operator $ I_{N,l,n}^i : V_N^{n+i} \rightarrow V_N^{n,l} $ as
		\begin{align*}
			I_{N,l,n}^i \left( \un{n+i}, \vn{n+i}, \VN{n+i} \right) = \left( \rv{\hun{n+i}}_{\Omega_{F}^{n,l}}, \left( \rv{\hun{n+i}}_{\Gamma^{n}} \cdot \nuF^n \right) \cdot \nuF^n + \left( \vn{n+i} \cdot \tauF^n \right) \cdot \tauF^n, \VN{n+i} \right),
		\end{align*}
		with space
		\begin{align*}
			V_N^{n,l} = \left\{ (\bu, \bv, \bV) \in \bH^{\onehalf}(\Omega_{F}^{n,l}) \times \bL^2(\Gamma) \times \bL^2(\Omega_{S}) : \nabla \cdot \bu = 0, \left( \rv{\bu}_{\Gamma^n} - \bv \right) \cdot \nuF^n = 0 \right\}.
		\end{align*}
		
		Finally, we complete the proof by following \cite{CGM2020} and \cite{MC2019JDE}.
	\end{proof}
	
	\begin{remark}
		In \cite[Theorem 3.2]{MC2019JDE}, there is another condition (A3) need to be verified, namely,
		$$
			\norm{\TN \uNn - \uNn}_{L^2(0,T;H)}^2 \leq \Dt.
		$$
		We note that this condition is not necessary, since it has been shown in \cite[Theorem 3.2]{MC2019JDE} that the conclusion of Theorem \ref{compact} is valid without (A3). The reasons why the authors left it in the main theorem is that this condition is usually satisfied as a by-product of Rothe's method and it nicely shows the role of numerical dissipation in compactness arguments while simplifying the proof.
	\end{remark}
	\begin{remark}
		In our work, the fluid is motioned by a generalized Non-Newtonian constitutive which results in a $ L^{p} $ regularity. However, in order to pass to the limits in the next section, we need the $ L^{2} $ strong convergence instead of the $ L^{p} $ strong convergence. This is because the convergence is used in the first and second convergence of equation \eqref{approximatepro} and it needs a $ L^{2} $ regularity.
	\end{remark}

	
	The compactness stated in Theorem \ref{compactnessuv} implies the following strong convergence results.
	\begin{corollary}
		As $ N \rightarrow \infty $, the following strong convergence results hold:
		\begin{enumerate}
			\item $ \uNn \rightarrow \bu $ in $ L^{2}(0,T; \bL^{2}(\Omega_{F})) $;
			\item $ \tvNn \rightarrow \bv $ in $ L^{2}(0,T; \bH^{-s}(\Gamma)) $, $ 0 < s < \onehalf $;
			\item $ \vNn \rightarrow \bv $ in $ L^{2}(0,T; \bH^{-s}(\Gamma)) $, $ 0 < s < \onehalf $.
		\end{enumerate}
	\end{corollary}
%
%
%

\subsubsection{Strong convergence for displacement and geometry parameters}
	In the following, we give the strong convergence of the thin structure displacement. To achieve this goal, we notice that $ \etaNn $ is uniformly bounded in $ W^{1, \infty}(0,T; \bL^2(\Gamma)) \cap \bL^{\infty} (0,T; \bH_0^2(\Gamma)) $, then from the continuous embedding
	$$
		W^{1, \infty}(0,T; \bL^2(\Gamma)) \cap \bL^{\infty} (0,T; \bH_0^2(\Gamma)) \hookrightarrow C^{0,1-\beta} ([0,T]; \bH^{2\beta} (\Gamma)), \quad 0 < \beta < 1,
	$$
	we have uniform boundedness of $ \etaNn $ in $ C^{0,1-\beta} ([0,T]; \bH^{2\beta} (\Gamma)) $. Due to the compact embedding of $ \bH^{2\beta} \hookrightarrow \bH^{2\beta -\epsilon} $ for every fixed $ \epsilon > 0 $ and the fact that functions in $ C^{0,1-\beta} ([0,T]; \bH^{2\beta} (\Gamma)) $ are uniformly continuous in time on finite interval, we find, by applying the Arzela-Ascoli Theorem, that as $ N \rightarrow \infty $
	$$
		\etaNn \rightarrow \bbeta \text{ in } C([0,T]; \bH^{2s}(\Gamma)), \quad 0 < s < 1,
	$$
	and
	$$
		\TN \etaNn \rightarrow \bbeta \text{ in } C([0,T]; \bH^{2s}(\Gamma)), \quad 0 < s < 1.
	$$
	 Then by the similar procedure in \cite[Lemma 3]{MC2013ARMA}, we have the  strong convergence results for structure displacement.
	\begin{theorem}\label{etacon}
		We have the following strong convergence results as $ N \rightarrow \infty $:
		\begin{enumerate}
			\item $ \etaNn \rightarrow \bbeta $ in $ L^{\infty}(0,T; \bH^{2s}(\Gamma)) $, $ 0 < s < 1 $;
			\item $ \TN\etaNn \rightarrow \bbeta $ in $ L^{\infty}(0,T; \bH^{2s}(\Gamma)) $, $ 0 < s < 1 $.
		\end{enumerate}
	\end{theorem}
	
	Consequently, considering our 2D fluid problem and 1D structure problem, we have $ H^{s}(\Gamma) \hookrightarrow C^1(\overline{\Gamma}) $ for $ s > \frac{3}{2} $. Theorem \ref{etacon} implies the following result.
	\begin{corollary}[Convergence for displacement]\label{etaconvergence}
		The following uniform convergence results hold as $ N \rightarrow \infty $:
		\begin{enumerate}
			\item $ \etaNn \rightarrow \bbeta $ in $ L^{\infty}(0,T; \bC^1(\overline{\Gamma})) $;
			\item $ \TN\etaNn \rightarrow \bbeta $ in $ L^{\infty}(0,T; \bC^1(\overline{\Gamma})) $.
		\end{enumerate}
	\end{corollary}
	
	To pass to the limit in the weak formulation, we still need the convergence of geometry parameters due to the effect of Navier-slip. In this sense, both normal and tangential structure displacements are considered to be non-zero. This may bring additional difficulties when take $ N \rightarrow \infty $. By means of above statements and the explicit formulas of the normals $ \nuFN $, the tangents $ \tauFN $ and quantities associated with $ \ANn $, we can deduce the corresponding strong convergence result as follows :
	\begin{corollary}[Convergence for geometry quantities, see also \cite{MC2016JDE}]\label{geoconvergence}
		For $ \nuFN $, $ \tauFN $ and quantities associated with $ \ANn $ as defined earlier, we have the following convergence as $ N \rightarrow \infty $:
		\begin{enumerate}
			\item $ \nuFN \rightarrow \nuF^{\bbeta} $ in $ L^{\infty}(0,T; C (\overline{\Gamma})) $;
			\item $ \tauFN \rightarrow \tauF^{\bbeta} $ in $ L^{\infty}(0,T; C (\overline{\Gamma})) $;
			\item $ \wNn \rightarrow \bw^{\bbeta} $ in $ L^{2}(0,T; H^1 (\Omega_{F})) $;
			\item $ J_{F,N} \rightarrow J_{F}^{\bbeta} $ in $ L^{\infty}(0,T; C (\overline{\Gamma})) $;
			\item $ \JNN \rightarrow \Jeta $ in $ L^{\infty}(0,T; C (\overline{\Omega}_{F})) $;
			\item $ \TN \JNN \rightarrow \Jeta $ in $ L^{\infty}(0,T; C (\overline{\Omega}_{F})) $;
			\item $  \left( \nabla \ANn \right)^{-1} \rightarrow \left( \nabla \bA_{\bbeta} \right)^{-1} $ in $ L^{\infty}(0,T; C (\overline{\Omega}_{F})) $.
		\end{enumerate}
	\end{corollary}

\section{The limiting problem}
	\label{limitproblem}
	In the first part of this section, we construct the suitable test functions that converge to the test funtions in weak formulation. Then, we pass to the limit of approximate problem by means of the weak and strong convergence results we obtained before.

\subsection{Construction of suitable test functions}
	Since the test functions in \eqref{testfunctionspace} for the limiting problem  depend on $ \bbeta $, we are in the position to construct the test functions for limiting problem and for the approximate problem due to the fact that test functions rely on the parameter $ N $.
	
	To eliminate the dependence of test functions on $ N $, we follow the same ideas proposed in \cite{CDEG2005,MC2014JDE}. Our goal is to restrict the space of all test functions $ \QE $ to a dense subset, which is denoted by $ \XE $. Then we construct a sequence of $ \qN $ of test functions such that for every $ \bq \in \XE $, $ \qN \rightarrow \bq $ in suitable norms. This idea has been used in different fluid-structure interaction problems, see e.g., \cite{CGM2020,MC2013ARMA,MC2014JDE,MC2016JDE,TW2018}.
	
	First, for $ n = 0, 1, \cdots, N - 1 $, we define the uniform domain which contains all the approximate domains as
	\begin{equation*}
		\OM = \bigcup_{\Dt > 0, n \in \mathbb{N}}
		\Omega_{F}^{\bbeta_{N}^{n}}.
	\end{equation*}
	Next, we introduce
	\begin{equation*}
		\XM =
		\left\{
			\br \in C_c^1 ([0,T); C^1(\overline{\Omega}_{\rm max})):
			\nabla \cdot \br = 0, \br \cdot \btau = 0, \text{ on } \Gamma_{\rm in/out}, \br \cdot \nuF = 0, \text{ on } \Gamma_{\rm b}
		\right\}
	\end{equation*}
	and
	\begin{equation*}
		\XE =
		\left\{
			(\bq, \bphi, \bpsi):
			\left| \
				\ad{
					& \bq(t, \cdot) = \rv{\br(t, \cdot)}_{\Omega_{F}^{\bbeta}(t)} \circ \bA_{\bbeta} (t),\ \br \in \XM, \\
					& \bpsi|_{\Gamma} = \bphi,\ \left( \rv{\br}_{\Gamma^{\bbeta}} - \bphi \right) \cdot \nuF^{\bbeta} = 0,\\
					& \bphi \in C_c^1 ([0,T); \bH_{0}^2(\Gamma)),\ \bpsi \in C_c^1 ([0,T); \bH_{0}^1(\Omega_{S}))
				}
			\right.
		\right\}.
	\end{equation*}
	It can be easily checked that $ \XE $ is dense in $ \QE $.
	
	Then we define the approximate test functions $ ( \qN, \phiN) $ in $ (n\Dt, (n+1)\Dt] $ as
	\begin{align*}
		& \qN(t, \cdot) = \qN^{n+1} : = \rv{\br ((n+1)\Dt, \cdot)}_{\Omega_F^{\etaNn}(t)} \circ \AN{n+1}(t), \\
		& \phiN(t) = \phiN^{n+1} : = \bphi((n+1)\Dt).
	\end{align*}
	It is clear that $ (\qN(t, \cdot), \phiN(t, \cdot), \bpsi(t, \cdot)) \in \W^{\bbeta}  $ for $ t \in (n\Dt, (n+1)\Dt] $. Fixing $ (\bq, \bphi, \bpsi) \in \XE $ with $ \bq(t, \cdot) = \rv{\br(t, \cdot)}_{\Omega_{F}^{\bbeta}(t)} \circ \bA_{\bbeta} (t) $, $ \br \in \XM $, we obtain the following lemma using the idea from \cite{MC2013ARMA,MC2014JDE,MC2016JDE}.
	
	\begin{lemma}[\cite{MC2016JDE}]\label{testconvergence}
		For every $ (\bq, \bphi, \bpsi) \in \XE $, we have
		\begin{enumerate}
			\item $ (\qN, \phiN) \rightarrow (\bq, \bphi) \text{ in } L^{\infty} (0,T; \bC^1(\overline{\Omega}_{F})) \times L^{\infty} (0,T; \bC^1(\overline{\Gamma})), $
			\item $ d\qN \rightarrow \pt \bq \text{ in } L^{2} (0,T; \bL^2(\Omega_{F})). $
		\end{enumerate}
	\end{lemma}

	\begin{lemma}[Convergence of gradients]
		\label{gradientscon}
		$ \bm{M} = \bD^{\bbeta} \bu $ and $ \bm{G} = \bbS(\bD^{\bbeta}(\bu)) $, where $ \bu $ and $ \bbeta $ are the weak* limits given by Lemma \ref{weakstarcon}, $ \bm{M} $ and $ \bm{G} $ are the weak limits given by Lemma \ref{weakcon}.
	\end{lemma}
	\begin{proof}
		As in \cite{CGM2020,MC2014JDE}, it is helpful to map the approximate fluid velocities and the limiting fluid velocity onto the physical domains. For that purpose, we introduce the following functions
		\begin{align*}
			& \bu^N(t, \cdot) = \uNn(t, \cdot) \circ \bA^{-1}_{\tetaNn}(t), && \tu(t, \cdot) = \bu(t, \cdot) \circ \bA_{\bbeta}^{-1}(t), \\
			& \chi^N \bm{g}(t, \bm{x}) =
			\left\{
				\ad{
					& \bm{g}, && \bm{x} \in \Omega_{F}^{\tetaNn}(t), \\
					& 0, && \bm{x} \notin \Omega_{F}^{\tetaNn}(t),
				}
			\right.
			&& \chi \bm{g}(t, \bm{x}) =
			\left\{
				\ad{
					& \bm{g}, && \bm{x} \in \Omega_{F}^{\bbeta}(t), \\
					& 0, && \bm{x} \notin \Omega_{F}^{\bbeta}(t),
				}
			\right.
		\end{align*}
		where $ \bA_{\tetaNn} $ is the ALE mapping that maps the reference domain $ \Omega_{F} $ into $ \Omega_{F}^{\tetaNn}(t) $, $ \bA_{\bbeta} $ is the ALE mapping defined in Section \ref{ALE}, $ \tetaNn $ is the time shift displacement defined in Section \ref{uniformboundedness} and $ \bbeta $ is the weak* limit in Lemma \ref{weakstarcon}. Due to the strong convergence of $ \etaNn $ and $ \tetaNn $ (that is $ \TN \etaNn$) in Corollary \ref{etaconvergence}, it is easy to find that as $ N \rightarrow \infty $,
		\begin{align}
			& \meas \left\{ \Omega_{F}^{\etaNn} \triangle \Omega_{F}^{\bbeta} \right\} \rightarrow 0, \label{measurezero}\\
			& \meas \left\{ \Omega_{F}^{\tetaNn} \triangle \Omega_{F}^{\bbeta} \right\} \rightarrow 0, \label{measurezerot}
		\end{align}
		where $ \triangle $ denote the symmetric difference of two sets, i.e., $ A \triangle B = (A \backslash B) \cup (B \backslash A) $ for two sets $ A $ and $ B $. According to the property of sets, we obtain from \eqref{measurezero} and \eqref{measurezerot} that
		\begin{align*}
			\meas \left\{ \Omega_{F}^{\tetaNn} \triangle \Omega_{F}^{\etaNn} \right\}
			\leq \meas \left\{ \Omega_{F}^{\etaNn} \triangle \Omega_{F}^{\bbeta} \right\}
			+ \meas \left\{\Omega_{F}^{\tetaNn} \triangle \Omega_{F}^{\bbeta} \right\}
			\rightarrow 0.
		\end{align*}
		Hence, $ \bA_{\tetaNn} = \ANn $ in the sense of limit and it follows from Corollary \ref{geoconvergence} that
		\begin{align} \label{detconverg}
			\frac{1}{\det \nabla \bA_{\tetaNn}} \rightarrow \frac{1}{\det \nabla \bA_{\bbeta}} \text{ strongly in } L^{\infty}(0,T; C(\overline{\Omega}_F)).
		\end{align}
		
		We follow the ideas in \cite[Proposition 7.6]{MC2014JDE} and \cite[Section 10.2]{CGM2020}, and provide a sketch of proof in the following three steps:
		
		\textbf{Step 1.} The strong convergence
		\begin{equation*}
			\chi^N \bu^N \rightarrow \chi \tu, \quad \text{ strongly in } L^2((0,T) \times \OM),
		\end{equation*}
		can be easily checked from \cite[Proposition 7.6]{MC2014JDE}, so we omit it here.
		
		\textbf{Step 2.} We need to prove
		\begin{equation*}
			\ad{
			& \chi^N \bD (\bu^N) \rightharpoonup \chi \bD (\tu), && \text{ weakly in } L^p((0,T) \times \OM), \\
			& \chi^N \bbS ( \bD (\bu^N) \left( \nabla \right) ) \rightharpoonup \chi \bbS ( \bD (\tu) ), && \text{ weakly in } L^q((0,T) \times \OM),
			}
		\end{equation*}
		where $ \bD (\bu^N) = \bD^{\tetaNn} (\uNn) $. From Lemma \ref{boundDU}, there exist $ \TM $ and $ \TG $, such that $ \chi^N \bD ( \bu^N ) \rightharpoonup \TM $ weakly in $ L^p((0,T) \times \OM)^2 $ and $ \chi^N \bbS ( \bD (\bu^N) ) \rightharpoonup \TG $ weakly in $ L^q((0,T) \times \OM)^2 $, i.e.,
		\begin{align*}
			& \intt \int_{\OM} \TM \cdot \by = \lim_{N \rightarrow \infty} \intt \int_{\OM} \chi^N \bD( \bu^N ) \cdot \by, & \by \in C_c^{\infty} ((0,T) \times \OM), \\
			& \intt \int_{\OM} \TG \cdot \by = \lim_{N \rightarrow \infty} \intt \int_{\OM} \chi^N \bbS (\bD( \bu^N )) \cdot \by, & \by \in C_c^{\infty} ((0,T) \times \OM).
		\end{align*}
		In order to obtain $ \TM = \chi \bD (\tu) $ and $ \TG = \chi \bbS (\bD (\tu)) $, we divide $ \OM $ into $ \Omega_{F}^{\bbeta} (t) $ and $ \OM \backslash \Omega_{F}^{\bbeta} (t) $. Taking the test function $ \by $ such that $ \by $ is supported in $ (0,T) \times \OM \backslash \Omega_{F}^{\bbeta} (t) $ and combining the uniform convergence of $ \tetaNn $ which derives \eqref{measurezerot}, we find that $ \TM = \TG = 0 $ in $ (0,T) \times (\OM \backslash \Omega_{F}^{\bbeta} (t)) $.
		
		Next, taking a test funtion $ \bz $ such that $ \text{supp} \bz \subset ((0,T) \times \Omega_{F}^{\bbeta} (t)) $ and combining the uniform convergence of $ \tetaNn = \TN \etaNn $, we get
		\begin{align*}
			& \intt \int_{\OM} \TM \cdot \bz = \lim_{N \rightarrow \infty} \intt \int_{\OM} \chi^N \bD( \bu^N ) \cdot \bz = \lim_{N \rightarrow \infty} \intt \int_{\Omega_{F}^{\bbeta}(t)} \bD( \bu^N ) \cdot \bz, \\
			& \intt \int_{\OM} \TG \cdot \bz = \lim_{N \rightarrow \infty} \intt \int_{\OM} \chi^N \bbS (\bD( \bu^N )) \cdot \bz = \lim_{N \rightarrow \infty} \intt \int_{\Omega_{F}^{\bbeta}(t)} \bbS (\bD( \bu^N )) \cdot \bz.
		\end{align*}
		Since the strong convergence $ \chi^N \bu^N \rightarrow \chi \tu $ in $ L^2((0,T) \times \OM) $ holds as shown in \textbf{Step 1}, we find that in the set $ \text{supp} \bz $, both $ \bu^N \rightarrow \tu $ and $ \bD (\bu^N) \rightarrow \bD (\tu) $ hold in the sense of distributions. Consequently, we have
		\begin{equation*}
			\ad{
				& \intt \int_{\OM} \TM \cdot \bz = \lim_{N \rightarrow \infty} \intt \int_{\Omega_{F}^{\bbeta}(t)} \bD( \bu^N ) \cdot \bz = \intt \int_{\Omega_{F}^{\bbeta}(t)} \bD (\tu) \cdot \bz,
			}
		\end{equation*}
		for all the test functions $ \bz $ supported in $ (0,T) \times \Omega_{F}^{\bbeta} (t) $. Due to the uniqueness of the limit, we obtain
		\begin{align*}
			& \TM = \bD (\tu) \text{ a.e. in } (0,T) \times \Omega_{F}^{\bbeta} (t).
		\end{align*}
		
		However, since the viscosity is nonlinear, if $ \bbS (\bD (\bu^N)) \rightarrow \bbS (\bD (\tu)) $ in the sense of distributions is a problem and thus, we can not proceed like $ \TM = \bD (\tu) $. We notice the $ p- $Laplacian structure of $ \bbS $ and combine the monotone operator theory to overcome this difficulty. More specifically, we use the ``Minty's trick'' to obtain the value of $ \TG $. There is still one problem that $ \TG $ is defined on moving domains $ \Omega_{F}^{\bbeta}(t) $, the gradients of velocities depend on the displacement $ \bbeta $. In this work, we introduce a localized Minty's Trick (see Proposition \ref{localMintytrick}), which contains of a cutoff function that can transfer the moving domain $ \Omega_{F}^{\bbeta}(t) $ to a fixing domain $ \OM $.
		
		Let $ \Omega = \OM $, $ \bu_m = \bu^N $, $ \bu = \tu $ and $ \tilde{\bbS} = \TG $ in Proposition \ref{localMintytrick}, then $ Q = (0,T) \times \OM $. We define $ \zeta_N (t, \bm{x}) $, $ \zeta(t, \bm{x}) $ as
		\begin{align*}
			& \zeta_N (t, \bm{x}) =
			\left\{
				\ad{
					& 1, && \bm{x} \in \Omega_{F}^{\etaNn}(t), \\
					& 0, &&\bm{x} \in \OM \backslash \Omega_{F}^{\etaNn}(t),
				}
			\right.
			&& \zeta (t, \bm{x}) =
			\left\{
			\ad{
				& 1, && \bm{x} \in \Omega_{F}^{\bbeta}(t), \\
				& 0, &&\bm{x} \in \OM \backslash \Omega_{F}^{\bbeta}(t).
			}
			\right.
		\end{align*}
		It is easy to check $ \zeta_N \rightarrow \zeta $ a.e. in $ Q $ as $ N \rightarrow \infty $, which means that \eqref{zetam} and \eqref{zetalim} are satisfied. Also, we have $ \bD(\uNn) \rightharpoonup \bD(\tu) $ in $ \bL^p((0,T) \times \OM)^2 $ and $ \bbS (\bD(\uNn)) \rightharpoonup \TG $ in $ L^q((0,T) \times \OM)^2 $ as obtained earlier. To verify \eqref{bseq} in Proposition \ref{localMintytrick}, we carry out a direct calculation to find
		\begin{align*}
			& \quad \abs{\int_{Q} \bbS (\bD (\bu^N)) : \bD(\bu^N) \zeta_N - \int_{Q} \TG : \bD(\tu) \zeta} \\
			& \leq \abs{\int_{Q} \left(  \bbS (\bD (\bu^N)) - \TG \right) : \bD (\bu^N) \zeta_N} + \abs{\int_{Q} \TG : \left( \bD (\bu^N) - \bD (\tu) \right) \zeta_N} \\
			& \quad + \abs{\int_{Q} \TG : \bD (\tu) \left( \zeta_N - \zeta \right)}.
		\end{align*}
		By the convergences of $ \bbS(\bD(\bu^N)) $, $ \bD(\bu^N) $ and $ \zeta_N $, we have \eqref{bseq}. Then from Proposition \ref{localMintytrick}, we achieve $ \TG \zeta = \bbS (\bD(\tu)) \zeta $ a.e. in $ Q $, which means
		$$
			\TG = \bbS (\bD(\tu)) \text{ a.e. in } (0,T) \times \Omega_{F}^{\bbeta}(t).
		$$
		
		\textbf{Step 3.} Finally, we are in the position to show that
		\begin{align*}
			& \intt \int_{\Omega_{F}} \bm{M} : \bq = \intt \int_{\Omega_{F}} \bD^{\bbeta}(\bu) : \bq, \\
			& \intt \int_{\Omega_{F}} \bm{G} : \bq = \intt \int_{\Omega_{F}} \bbS(\bD^{\bbeta}(\bu)) : \bq,
		\end{align*}
		for every test function $ (\bq, 0, 0) \in \XE $. It follows from the results of \textbf{Step 2}, the uniform boundedness and convergence of gradients $ \bD^{\tetaNn} (\tun) $ and $ \bbS( \bD^{\tetaNn} (\tun) ) $ provided by Lemma \ref{boundDU} and \ref{weakcon}, the strong convergence of $ \left( \det \nabla \bA_{\tetaNn} \right)^{-1} $ given in \eqref{detconverg} and the strong convergence of the test functions $ \qN \rightarrow \bq $ obtained in Lemma \ref{testconvergence} that
		\begin{align*}
			\intt \int_{\Omega_{F}} \bm{M} : \bq
			& = \lim_{N \rightarrow \infty} \intt \int_{\Omega_{F}} \bD^{\tetaNn} (\uNn) : \qN \\
			& = \lim_{N \rightarrow \infty} \intt \int_{\Omega_{F}} \bD (\uNn \circ \bA_{\tetaNn}^{-1}) : \left( \br \circ \ANn \right) \rd \Omega_{F} \\
			& = \lim_{N \rightarrow \infty} \intt \int_{\Omega_{F}^{\tetaNn}} \bD (\bu^N) : \br \rd \Omega_{F}^{\tetaNn} \cdot \frac{1}{\det \nabla \bA_{\tetaNn}} \\
			& = \lim_{N \rightarrow \infty} \intt \int_{\OM} \frac{1}{\det \nabla \bA_{\tetaNn}} \chi^N \bD (\bu^N) : \br \rd \OM \\
			& = \intt \int_{\Omega_{F}^{\bbeta}} \frac{1}{\det \nabla \bA_{\bbeta}} \bD (\tu) :\br \rd \Omega_{F}^{\bbeta} \\
			& = \intt \int_{\Omega_{F}} \bD^{\bbeta} (\bu) : \bq,
		\end{align*}
		\begin{align*}
			\intt \int_{\Omega_{F}} \bm{G} : \bq
			& = \lim_{N \rightarrow \infty} \intt \int_{\Omega_{F}} \bbS (\bD^{\tetaNn} (\uNn)) : \qN \\
			& = \lim_{N \rightarrow \infty} \intt \int_{\Omega_{F}} \bbS (\bD (\uNn \circ \bA_{\tetaNn}^{-1})) : \left( \br \circ \ANn \right) \rd \Omega_{F} \\
			& = \lim_{N \rightarrow \infty} \intt \int_{\Omega_{F}^{\tetaNn}} \bbS (\bD (\bu^N)) : \br \rd \Omega_{F}^{\tetaNn} \cdot \frac{1}{\det \nabla \bA_{\tetaNn}} \\
			& = \lim_{N \rightarrow \infty} \intt \int_{\OM} \frac{1}{\det \nabla \bA_{\tetaNn}} \chi^N \bbS (\bD (\bu^N)) : \br \rd \OM \\
			& = \intt \int_{\Omega_{F}^{\bbeta}} \frac{1}{\det \nabla \bA_{\bbeta}} \bbS (\bD (\tu)) :\br \rd \Omega_{F}^{\bbeta} \\
			& = \intt \int_{\Omega_{F}} \bbS (\bD^{\bbeta} (\bu)) : \bq,
		\end{align*}
		where we used from \eqref{gradequiv} that $ \bD (\bu^N) = \bD^{\tetaNn} (\uNn) $ and $ \bD (\tu) = \bD^{\bbeta} (\bu) $. This completes the proof.
	\end{proof}
	
	We introduce the localized Minty's Trick here from \cite[Appendix A]{Wolf2007}.	
	\begin{proposition}[Localized Minty's Trick {\cite[Appendix A]{Wolf2007}}]\label{localMintytrick}
		Let $ \bu_m \in L^p(0,T; \bW^{1,p}(\Omega)) $ and $ \zeta_m \in L^{\infty}(Q) $ with $ Q = (0,T) \times \Omega $. If for $ a_0 = \text{const} > 0 $,
		\begin{align}
		& 0 \leq \zeta_m \leq a_0 \text{ a.e. in } Q, \quad m \in \mathbb{N}, \label{zetam}\\
		& \bD(\bu_m) \rightharpoonup \bD(\bu) \text{ weakly in } \bL^{p}(Q)^2, \label{dum}\\
		& \bbS ( \bD(\bu_m) ) \rightharpoonup \tilde{\bbS} \text{ weakly in } \bL^{q}(Q)^2, \label{bsm}\\
		& \zeta_m \rightarrow \zeta \text{ a.e. in } Q \text{ as } m \rightarrow \infty, \label{zetalim}\\
		& \limsup_{m \rightarrow \infty} \int_Q \bbS(\bD(\bu_m)) : \bD(\bu_m) \zeta_m = \int_Q \tilde{\bbS} : \bD(\bu) \zeta, \label{bseq}
		\end{align}
		Then
		\begin{equation}\label{tildeS}
		\tilde{\bbS} \zeta = \bbS(\bD(\bu)) \zeta \text{ a.e. in } Q.
		\end{equation}
	\end{proposition}
	
	By the analogous argument above, it is easy to deduce the following corollary.
	\begin{corollary}[\cite{MC2014JDE}]
		For every $ (\bq, \bphi, \bpsi) \in \XE $, we have
		\begin{itemize}
			\item[] $ \bD^{\etaNn} (\qN) \rightarrow \bD (\bq) \text{ in } L^{p} (0,T; \bL^p(\Omega_{F}))^2. $
		\end{itemize}
	\end{corollary}
	
\subsection{Pass to the limit}
	To get the weak formulation of the coupled problem, setting $ (\phiN, \bpsi) $ as the test functions in \eqref{weaks} and integrating it over $ (n\Dt, (n+1)\Dt) $, taking $ (\qN,\phiN) $ as the test functions in \eqref{weakf}, multiplying $ \frac{1}{\Dt} $, again integrating it over $ (n\Dt, (n+1)\Dt) $, and adding the two equations together, we have
	\begin{align}
		& \quad \inttt \int_{\Omega_{F}} \JNn \frac{\un{n+1} - \un{n}}{\Dt} \cdot \qN
		+ \onehalf \inttt \int_{\Omega_{F}} \frac{\JN{n+1} - \JN{n}}{\Dt} \un{n+1} \cdot \qN \nonumber \\
		& + \onehalf \inttt \int_{\Omega_{F}} \JN{n} \biggl( \left( \left( \un{n} - \wn{n+1} \right) \cdot \nabla^{n+1} \right) \un{n+1} \cdot \qN \bigg. \nonumber \\
		& \bigg. \qquad \qquad \qquad \qquad \qquad \qquad \quad - \left( \left( \un{n} - \wn{n+1} \right) \cdot \nabla^{n+1} \right) \qN \cdot \un{n+1} \biggr) \nonumber \\
		& + 2 \inttt \int_{\Omega_{F}} \JN{n} \bbS(\bD(\un{n+1})) : \bD(\qN) \label{approximatedt} \\
		& + \frac{1}{\alpha} \inttt \int_{\Gamma} \left( \bu_{N,\tauFN^{n+1}}^{n+1} - \bv_{N,\tauFN^{n+1}}^{n+1} \right) \left( \bq_{N, \tauFN^{n+1}} - \bphi_{N, \tauFN^{n+1}} \right) J_{F,N}^{n+1} \nonumber \\
		& + \inttt \int_{\Gamma} \frac{ \vn{n+1} - \vn{n} }{\Dt} \cdot \phiN + \inttt \int_{\Omega_{S}} \frac{ \VN{n+1} - \VN{n} }{\Dt} \cdot \bpsi \nonumber \\
		& + \inttt \inner{\LE \etan{n+1}}{\phiN} + \inttt \as{\dNn^{n+1}}{\psiN}
		+ \inttt \inner{f(\etan{n+1})}{\phiN} \nonumber \\
		& = \inttt \int_{\Gamma_{\rm in/out}} P_{\rm in/out}^{n} \qN \cdot \nuF. \nonumber
	\end{align}
	Summing \eqref{approximatedt} from $ n = 0, 1 , \dots, N - 1 $, we obtain the weak formulation of approximate problem over $ (0,T) $ as
	\begin{align}
		& \quad \intt \int_{\Omega_{F}} \TN \JNN \pt \stun \cdot \qN
		+ \onehalf \intt \int_{\Omega_{F}} \frac{\JNN - \TN \JNN}{\Dt} \uNn \cdot \qN \nonumber\\
		& + \onehalf \intt \int_{\Omega_{F}} \JNN \biggl( \left( \left( \TN \uNn - \wNn \right) \cdot \nabla^{\etaNn} \right) \uNn \cdot \qN \bigg. \nonumber\\
		& \bigg. \qquad \qquad \qquad \qquad \quad - \left( \left( \TN \uNn - \wNn \right) \cdot \nabla^{\etaNn} \right) \qN \cdot \uNn \biggr) \nonumber\\
		& + 2 \intt \int_{\Omega_{F}} \JNN \bbS(\bD^{\etaNn}(\uNn)) : \bD^{\etaNn}(\qN) \label{approximatepro}\\
		& + \frac{1}{\alpha} \intt \int_{\Gamma} \left( \bu_{N,\tauFN} - \bv_{N,\tauFN} \right) \left( \bq_{N, \tauFN} - \bphi_{N, \tauFN} \right) J_{F,N} \nonumber\\
		& + \intt \int_{\Gamma} \pt \tvn \cdot \phiN + \intt \int_{\Omega_{S}} \pt \tVn \cdot \bpsi \nonumber\\
		& + \intt \inner{\LE \etaNn}{\phiN} + \intt \as{\dNn}{\bpsi}
		+ \intt \inner{f(\etaNn)}{\phiN} \nonumber\\
		& = \intt \int_{\Gamma_{\rm in/out}} P_{\rm in/out} \qN \cdot \nuF, \nonumber
	\end{align}
	where $ \stun $, $ \tvn $ and $ \tVn $ are the piece-wise linear approximations of $ \uNn $, $ \vNn $ and $ \VNn $, that is for $ t \in (n\Dt, (n+1)\Dt] $,
	\begin{align*}
		& \stun (t) = \un{n} + \frac{\un{n+1} - \un{n}}{\Dt} (t - \Dt), \\
		& \tvn (t) = \vn{n} + \frac{\vn{n+1} - \vn{n}}{\Dt} (t - \Dt), \\
		& \tVn (t) = \VN{n} + \frac{\VN{n+1} - \VN{n}}{\Dt} (t - \Dt).
	\end{align*}
	
	In the sequel, we will take the limit $ N \rightarrow \infty $, which means $ \Dt = \frac{T}{N} \rightarrow 0 $. Here, we denote \eqref{approximatepro} as $ \sum_{i = 1}^{10} I_{i} = I_{11} $ and pass to the limit for each term.
	
	\begin{enumerate}
		\item $ I_1 $ and $ I_2 $: Using integration by parts and the convergence results earlier, we obtain
		\begin{align*}
			& \intt \int_{\Omega_{F}} \TN \JNN \pt \stun \cdot \qN + \intt \int_{\Omega_{F}} \frac{\JNN - \TN \JNN}{\Dt} \uNn \cdot \qN \\
			& \rightarrow - \intt \int_{\Omega_{F}} \Jeta \bu \cdot \pt \bq - \intt \int_{\Omega_{F}} \Jeta \left( \nabla^{\bbeta} \cdot \bw^{\bbeta} \right) \bu \cdot \bq - \int_{\Omega_{F}} J_0 \bu_{0} \bq (0).
		\end{align*}
		More details can be found in \cite[Proposition 8]{MC2016JDE}.
		\item $ I_3 $: We need to show the convergence of each term in $ I_3 $. First, the convergence of $ \TN \JNN $, $ \uNn $, $ \wNn $, $ \qN $, $ \nabla^{\tetaNn} \uNn $ and $ \nabla^{\tetaNn} \qN $ can be derived directly from the previous Lemmas. From the estimate 3 in Lemma \ref{estunif}, we have
		\begin{equation*}
			\norm{\TN \uNn - \uNn}_{L^2(0,T; L^{2}(\Omega))}^2 \leq C \sum_{n = 1}^{N} \norm{\un{n+1} - \un{n}}_{L^{2}(\Omega)}^2 \Dt \leq C \Dt,
		\end{equation*}
		which means $ \TN \uNn \rightarrow \bu $ in $ L^2(0,T; L^{2}(\Omega)) $. Therefore, we obtain the convergence of $ I_3 $.
		\item $ I_4 $: From the convergence of $ \JNN $, $ \bbS (\bD^{\etaNn}(\uNn)) $ and $ \bD^{\etaNn} (\qN) $ in some appropriate function spaces, we get
		\begin{align*}
			& \quad \intt \int_{\Omega_{F}} \JNN \bbS (\bD^{\etaNn}(\uNn)) : \bD^{\etaNn} (\qN)
			- \intt \int_{\Omega_{F}} \Jeta  \bbS (\bD^{\bbeta}(\bu)) : \bD^{\bbeta} (\bq) \\
			& = \intt \int_{\Omega_{F}} \JNN \left( \bbS (\bD^{\etaNn}(\uNn)) - \bbS (\bD (\bu)) \right) : \bD^{\etaNn} (\qN) \\
			& \quad + \intt \int_{\Omega_{F}} \left( \JNN - \Jeta \right) \bbS (\bD (\bu)) : \bD^{\etaNn} (\qN) \\
			& \quad + \intt \int_{\Omega_{F}} \Jeta \bbS (\bD (\bu)) : \left( \bD^{\etaNn} (\qN) - \bD (\bq) \right) \\
			& \rightarrow 0, \text{ as } N \rightarrow \infty.
		\end{align*}
		\item $ I_5 $: It is easy to get the convergence of $ I_5 $ by using the convergence of $ \uNn $, $ \vNn $ and the geometric quantities in Corollary \ref{geoconvergence}.
		\item $ I_6 $ and $ I_7 $: The convergence of $ \vsNn $ and $ \VsNn $ and integration by parts imply
		\begin{align*}
			& \intt \int_{\Gamma} \pt \vsNn  \cdot \phiN \rightarrow - \intt \int_{\Gamma} \bv \cdot \pt \bphi - \int_{\Gamma} \bv_{0} \cdot \bphi(0), \\
			& \intt \int_{\Omega_{S}} \pt \VsNn \cdot \psiN \rightarrow - \intt \int_{\Omega_{S}} \bV \cdot \pt \bpsi - \int_{\Omega_{S}} \bV_{0} \cdot \bpsi(0).
		\end{align*}
		\item $ I_8 $ and $ I_9 $: From the convergence of $ \etaNn $, $ \dNn $ and $ \phiN $ in some proper spaces, we can deduce the convergence of $ I_9 $ and $ I_{10} $ directly.
		\item $ I_{10} $: Since $ f $ is locally Lipschitz from $ \bH^{2-\epsilon} $ to $ H^{-2} $, it follows from the convergence of $ \etaNn $ and $ \phiN $ that
		\begin{align*}
			& \quad \intt \inner{f(\etaNn)}{\phiN} - \intt \inner{f(\bbeta)}{\bphi} \\
			& = \intt \inner{f(\etaNn) - f(\bbeta)}{\phiN} + \intt \inner{f(\bbeta)}{\phiN - \bphi} \\
			& \leq \intt \norm{f(\etaNn) - f(\bbeta)}_{H^{2} (\Gamma)} \norm{\phiN}_{\bH^{-2} (\Gamma)} + \intt \norm{f(\bbeta)}_{L^{2}(\Gamma)} \norm{\phiN - \bphi}_{\bL^{2} (\Gamma)} \\
			& \leq C \intt \norm{\etaNn - \bbeta}_{\bH^{2-\epsilon}(\Gamma)} \norm{\phiN}_{\bL^{2} (\Gamma)} + \intt \norm{f(\bbeta)}_{L^{2}(\Gamma)} \norm{\phiN - \bphi}_{\bL^{2} (\Gamma)} \\
			& \rightarrow 0, \text{ as } N \rightarrow \infty.
		\end{align*}
		\item $ I_{11} $: The convergence of $ \qN $ leads to the convergence of $ I_{11} $.
	\end{enumerate}
	
	Therefore, we have shown that the limiting functions $ \bu $, $ \bbeta $ and $ \bd $ as $ N \rightarrow \infty $ satisfy the weak form of original problem in the sense of \eqref{weakformulation} in Definition \ref{weaksolution}, for all test functions $ \bq $, $ \bphi $ and $ \psi $, which are dense in the test space $ \QE $. This means that the approximate solutions we constructed converge to a weak solution of problem \eqref{NS}--\eqref{eta0}.
	
	Now, we prove Theorem \ref{mainresult}.
	
	\begin{proof}[Proof of Theorem \ref{mainresult}]
		From the above analysis, we have the existence of the weak solution. To derive \eqref{energyestimate}, we take the limit of estimates 1 and 2 in Lemma \ref{estunif}. Thanks to the semi-continuity properties of norms, we recover the energy estimate in \eqref{energyestimate}.
	\end{proof}

\section*{Acknowledgments}
	This work was supported by the National Natural Science Foundation of China [grant number 11771216], the Key Research and Development Program of Jiangsu Province (Social Development) [grant number BE2019725], the Six Talent Peaks Project in Jiangsu Province [grant number 2015-XCL-020] and the Qing Lan Project of Jiangsu Province.

\end{document}